\documentclass[draft]{elsarticle}
\usepackage{amsmath}
\usepackage{amsthm}
\usepackage{amssymb}
\usepackage{graphics}
\usepackage{graphicx}
\usepackage{bm}
\usepackage{setspace}
\usepackage{stmaryrd,xspace}
\usepackage{tikz-cd}
\usepackage{adjustbox}
\usepackage[utf8]{inputenc}

\usepackage{amsfonts}
\usepackage{eucal}
\usepackage{latexsym}
\usepackage{mathdots}
\usepackage{mathrsfs}
\usepackage{enumitem}
\usepackage{stmaryrd,xspace}
\usepackage{footmisc}

\usepackage{color}

\def\nt{\mbox{\rotatebox[origin=c]{180}{$\stackrel{\leftarrow}{\scriptstyle \sphericalangle}$}}}

\newcommand{\bee}{\begin{equation}}
\newcommand{\be}{\begin{equation}}
\newcommand{\ee}{\end{equation}}
\newcommand{\ba}{\begin{eqnarray}}
\newcommand{\ea}{\end{eqnarray}}
\newcommand{\bi}{\begin{itemize}}
\newcommand{\ei}{\end{itemize}}
\newcommand{\bn}{\begin{enumerate}}
\newcommand{\en}{\end{enumerate}}
\newcommand{\bbm}{\begin{bmatrix}}
\newcommand{\ebm}{\end{bmatrix}}
\newcommand{\bp}{\begin{proof}}
\newcommand{\ep}{\end{proof}}

\newcommand{\mr}{\ensuremath{\mathrm}}
\newcommand{\scr}{\ensuremath{\mathscr}}

\newcommand{\mc}{\ensuremath{\mathcal}}

\newcommand{\ov}{\ensuremath{\overline}}
\newcommand{\sm}{\ensuremath{\setminus}}

\newcommand{\ga}{\ensuremath{\gamma}}
\newcommand{\Om}{\ensuremath{\Omega}}

\newcommand{\la}{\ensuremath{\lambda }}

\renewcommand{\a}{\ensuremath{\alpha }}

\def\C{\mathbb{C}}
\def\R{\mathbb{R}}
\def\D{\mathbb{D}}
\def\T{\mathbb{T}}

\def\N{\mathbb{N}}

\newcommand{\bmu}{\mu}

\renewcommand{\H}{\ensuremath{\mathcal{H} }}

\newcommand{\intfty}{\ensuremath{\int _{-\infty} ^{\infty}} }

\newcommand{\tr}{\operatorname{tr}}

\newcommand{\re}{\operatorname{Re}}

\newcommand{\ip}[2]{\ensuremath{\left \langle {#1} , {#2}\right \rangle}}

\newcommand{\dom}[1]{\ensuremath{\mathrm{Dom} ({#1}) }}
\newcommand{\spn}[1]{\ensuremath{\mathrm{Span} ({#1}) }}
\renewcommand{\dim}[1]{\ensuremath{\mathrm{dim} \left( {#1} \right) }}
\newcommand{\ran}[1]{\ensuremath{\mathrm{Ran} \left( {#1} \right) }}

\renewcommand{\ker}[1]{\ensuremath{\mathrm{Ker} ({#1}) }}

\newcommand{\ci}[1]{_{ {}_{\scriptstyle #1}}}
\newcommand{\ti}[1]{_{\scriptstyle \text{\rm #1}}}

\numberwithin{equation}{section}
\numberwithin{subsection}{section}

\newtheorem{thm}{Theorem}

\newtheorem{lemma}{Lemma}
\newtheorem{prop}{Proposition}
\newtheorem{cor}{Corollary}
\newtheorem{thm*}{Theorem}

\theoremstyle{definition}
\newtheorem{defn}{Definition}
\theoremstyle{remark}
\newtheorem{remark}{Remark}
\newtheorem*{remark*}{Remark}

\newcommand{\bga}{\gamma}
\newcommand{\sgn}{\text{sgn}}
\usepackage{mathtools}
\usepackage{xcolor}
\definecolor{ideas}{RGB}{0,100,0}
\definecolor{michael}{cmyk}{0.53,0.21,0,0.24}
\newcommand{\expp}[1]{\text{Exp} \left ( #1 \right )}

\makeatletter
\renewcommand*\env@matrix[1][\arraystretch]{%
  \edef\arraystretch{#1}%
  \hskip -\arraycolsep
  \let\@ifnextchar\new@ifnextchar
  \array{*\c@MaxMatrixCols c}}
\makeatother

\begin{document}

\title{Spectral Measures for Derivative Powers via Matrix-Valued Clark Theory}

\author[add1]{Michael Bush\corref{cor1}%
\fnref{fn1}}
\ead{mikebush@udel.edu}

\author[add1]{Constanze Liaw\fnref{fn1,fn2,fn3}}
\ead{liaw@udel.edu}

\author[add2]{Robert T.W. Martin\fnref{fn4}}
\ead{Robert.Martin@umanitoba.ca}

\cortext[cor1]{Corresponding Author}
\fntext[fn1]{The work of M.~Bush and C.~Liaw was partially supported by the National Science Foundation under the grant DMS-1802682.}
\fntext[fn2]{C.~Liaw was also affiliated with the Center for Astrophysics, Space Physics \& Engineering Research (CASPER) at Baylor University during the period when this work was done.}
\fntext[fn3]{Since August 2020, C.~Liaw has been serving as a Program Director in the Division of Mathematical Sciences at the National Science Foundation (NSF), USA, and as a component of this position, she received support from NSF for research, which included work on this paper. Any opinions, findings, and conclusions or recommendations expressed in this material are those of the authors and do not necessarily reflect the views of the NSF}
\fntext[fn4]{R.T.W.~Martin was partially supported by NSERC grant 2020-05683.}

\address[add1]{University of Delaware Department of Mathematical Sciences, 15 Orchard Rd., Newark, DE, United States}
\address[add2]{University of Manitoba Department of Mathematics, 186 Dysart Road University of Manitoba, Winnipeg, MB R3T 2N2, Canada}


\bibliographystyle{plain}
\begin{abstract}
The theory of finite-rank perturbations allows for the determination of spectral information for broad classes of operators using the tools of analytic function theory. In this work, finite-rank perturbations are applied to powers of the derivative operator, providing a full account from self-adjoint boundary conditions to computing aspects of the operators' matrix-valued spectral measures. In particular, the support and weights of the Clark (spectral) measures are computed via the connection between matrix-valued contractive analytic functions and matrix-valued nonnegative measures through the Herglotz Representation Theorem.
For operators associated with several powers of the derivative, explicit formulae for these measures are included. While eigenfunctions and eigenvalues for these operators with fixed boundary conditions can often be computed using direct methods from ordinary differential equations, this approach provides a more complete picture of the spectral information.
\end{abstract}

\begin{keyword}
finite-rank perturbations \sep matrix-valued spectral measures \sep self-adjoint extensions \sep Clark theory
\MSC[2020]{Primary - 34L05, 46N20, 47B25; Secondary - 46E22, 47B32}
\end{keyword}

\maketitle

\setcounter{tocdepth}{1}
\tableofcontents

\section{Introduction}
\label{ss-history}
Interest in rank-one perturbations was initiated by Weyl~\cite{Weyl}, who implemented them to determine spectral properties of Sturm--Liouville operators subject to changing a boundary condition. Consequently,  Aronszajn--Donoghue (see e.g.~\cite{TraceIdeals}) and Alexandrov--Clark (see e.g.~\cite{RossCT}) theory were established;  their investigations of the boundary values of the Cauchy transform ultimately resulted in a good understanding of these operators' spectral properties. 

Higher rank perturbations arise from changing several boundary conditions simultaneously. This naturally leads to a matrix-valued version of Aleksandrov--Clark theory, which was developed by Gesztesy--Tsekanovskii in~\cite{Gesztesy2000}, Kapustin--Poltoratski~\cite{KP2006} and Liaw--Martin--Treil~\cite{LMT}, among many others. The absolutely continuous parts of matrix-valued Aleksandrov--Clark measures were successfully studied as early as Kuroda's work in the context of scattering theory \cite{Kuroda1963}.
In a series of works (see e.g.~\cite{AalbeverioKurasov2000}), Albeverio--Kurasov carried out applications of finite-rank perturbations to certain self-adjoint extensions of symmetric operators with finite deficiency indices. Their main focus was the scattering theory of singular form bounded---and even more singular---finite-rank perturbations. They were not interested in explicit connections with boundary conditions.

In Aleman--Martin--Ross~\cite{AMR}, changing several boundary conditions in a \\ Sturm--Liouville operator was brought in connection with a de~Branges--Rovnyak model space from Aleksandrov--Clark theory. This was accomplished through presenting a tractable formula for the Liv{\v s}ic characteristic function (which generates the model space) of a simple symmetric operator. While the spectral information was encoded within, they were mainly interested in developing a general representation theory.

Here, we supplement the approach of Aleman--Martin--Ross~\cite{AMR} using results from Liaw--Martin--Treil~\cite{LMT} and the generalizations of Glazman--Krein--Naimark's classical self-adjoint extension theory by Sun et al.~\cite{BSHZ2019} and Wang et al.~\cite{WSZ2009}. We make concrete the connection between the Lagrange bracket (a sesquilinear form intrinsic to the nature of the operator) and the perturbation parameter; note that proofs from rank-one and rank-two cases do not immediately generalize to the higher rank setting. This allows us to paint a complete picture from boundary conditions to properties at the level of matrix-valued spectral information of symmetric differential expressions of any order with equal deficiency indices that satisfy what we call the patchwork property. As examples, we  explicitly carry out this approach for some powers of derivative operators, obtaining full matrix-valued spectral information. Included are operators defined by considering expressions of the form $ ( -1 )^{n}\frac{d^{2n}}{dx^{2n}}$ with $n\in\N$ on $L^2(0,\infty)$ acting on particular domains, for which the endpoint at infinity is non-regular. While the operators (i.e. expressions paired with domains) considered in this work are rather rudimentary, the main contribution consists of an explicit proof of concept: this framework can be applied to bring boundary conditions in direct relation with matrix-valued spectral properties.

A related way to investigate spectral questions under changing boundary conditions is via boundary triplets~\cite{BHdS2020}, a rather general framework, which is in principle very powerful. For example, it is connected with Dirichlet-to-Neumann maps from partial differential equations. Apart from Bush--Frymark--Liaw \cite{BFL}, only scalar (i.e. not matrix-valued) spectral information was obtained and the connection to boundary conditions has not been worked out very explicitly. Nonetheless, these works suggest that we should expect the class of operators for which our current approach is feasible to be much larger.

\subsection{Outline}
\label{ss-outline}
In Section~\ref{s-prelim}, we introduce notations and conventions, including the Cayley transform, classical self-adjoint extension theory and the higher order operators of interest (in Equations~\eqref{e-KnDef} and~\eqref{e-LnDef}). For the differential expression $-\frac{d^2}{dx^2}$ on the half-line, we present the connection between self-adjoint extensions boundary conditions and a rank-one perturbations in  Proposition~\ref{p-PertBCR1HalfLine}. Similarly, for the differential expression $i\frac{d}{dx}$ on an interval symmetric about the origin, we present this connection in Proposition~\ref{p-PertBCR1Interval}. This simple calculation serves as a motivation for the rank two and higher rank connections we consider in Sections~\ref{s-RankTwoExtensions} and~\ref{s-FiniteRankExtensions}, respectively.

The focus of Section~\ref{s-matrixclark} is a translation of results on matrix-valued Clark measures (from the circle) to the real line; specifically, the absolutely continuous part in Theorem~\ref{mtxAC} and the pure-point part in Theorem~\ref{t-Nevanlinna}. These results will later be used to investigate the spectrum of operators whose self-adjoint extensions are of a rank greater than or equal to one. A key step in the method we establish in Section~\ref{s-matrixclark} involves choosing bases for particular spaces tied to the problem. Theorems~\ref{t-LivsicEquivAC} and~\ref{t-LivsicEquivPP}, along with the discussion in Subsection~\ref{ss-FRCharacteristic}, address the well-posedness of the method with respect to this choice.

The goal of Section ~\ref{s-FiniteRankExtensions} is to establish methods for parameterizing self-adjoint extensions of operators with  deficiency indices that are finite and greater than one. Working on a wide class of operators that includes our operators of interest, we use the tools of symmetric differential expression theory to  describe self-adjoint extensions of these operators (Theorems~\ref{T-5WSZ2009} and \ref{T-1.1BSHZ2019}). The connection between the boundary conditions and the finite-rank perturbation parameter is contained in (Theorems~\ref{T_FRAlphaBCCOnnection2} and~\ref{T_FRAlphaBCCOnnection}), generalizing Propositions ~\ref{p-PertBCR1HalfLine} and \ref{p-PertBCR1Interval}. The even powers of the derivative operator on the half-line are discussed in Subsection~\ref{ss-SABCHalfLine}, while powers of the derivative on a symmetric interval about the origin can be found in~\ref{ss-SABCInterval}.

Lastly, we present spectral results for the rank-one and rank-two cases for powers of the derivative on the half-line (Section~\ref{s-HalfLineSpectrum}) and an interval symmetric about the origin (Section~\ref{s-IntervalSpectrum}). The primary results of Section~\ref{s-HalfLineSpectrum} are Theorem~\ref{t-HalfLineRankOneSpectrum} (rank-one) and Subsection~\ref{ss-HalfLineRankTwo} (rank-two). 

Similarly, in   Section~\ref{s-IntervalSpectrum} the main results are Theorems~\ref{thcarrierinterval1} and~\ref{th1} for the rank-one case and Theorem~\ref{t-IntervalRankTwoSpectrum} for the rank-two case. It is of note that we focus on the absolutely continuous parts of the spectral measures for the operators in Section~\ref{s-HalfLineSpectrum}, while the operators of Section~\ref{s-IntervalSpectrum} have spectral measures with only pure-point parts. 

\section{Preliminaries}
\label{s-prelim}
\subsection{Notation and Conventions}
\label{ss-notation}
First, we introduce two spaces of matrices that will we will use:

\begin{defn}
Let $\scr{U}_n$ denote the $n \times n$ complex-valued unitary matrices.
\end{defn}

\begin{defn}
Let $X \subset \C$ be open. The $n \times n$ contractive matrix-valued analytic functions on $X$ will be denoted by $\scr{S} _n (X)$,
\[\scr{S} _n (X):= [H^\infty (X) \otimes \C ^{n\times n} ] _1.\]
\end{defn}
Next, we define a notion of matrix symmetry that will be integral to setting the boundary conditions for our differential operators of interest. 
\begin{defn}
Let $M,C \in \C^{n \times n}$. We say that $M$ is {\bf{$\boldsymbol{C}$-symmetric}} if \[M=-C^{-1}M^{*}C.\]
\end{defn}
Another key tool in the setup of the problems we will be considering is the {\bf{Cayley Transform}}. 

\begin{defn}
The Cayley Transform $\ga : \C ^+ \rightarrow \D$ is the fractional linear or M\"{o}bius transformation defined by
$$ \ga (w) := \frac{w-i}{w+i}. $$ The Cayley Transform, $\ga$ is a bijection of $\C ^+$ onto $\D$ with compositional inverse
$$ \ga ^{-1} (z) = i \frac{1+\zeta}{1-\zeta}. $$ 
\end{defn} 
Observe that $\ga (i) =0$, and that $\ga$ extends to a bijection of $\R$ onto $\scr{U}_1 \sm \{ 1 \}$, and hence of $\ov{\C ^+}$ onto $\ov{\D} \sm \{ 1 \}$. Motivated by the above definition for scalars, we now define the Cayley transform of a contractive operator $A$ as follows.
\begin{defn}
Analogously, we define the action of the Cayley Transform on an operator $\bga$ to be: $$\bga A := (A-iI)(A+iI)^{-1},$$ and its compositional inverse via
\[\bga^{-1} A := i(I+A)(I-A)^{-1}.\]
\end{defn}

\subsection{Von Neumann's theory of self-adjoint extensions}\label{ss-Rob}
This section is based on \cite[Chapter VII]{Glazman}.
Let $T$ be a linear expression that defines a closed, densely-defined, and symmetric operator on its domain $\dom{T} \subseteq \H$ with equal deficiency indices $(n,n)$.\footnote{Note that we will use the symbol $T$ to denote both expression and the operator defined by that expression acting on $\dom{T}$.} Thus, we can define the defect spaces,
\begin{equation}
\label{e-defect}
    \mc{D} _\pm := \ker{T^* \mp i I} = \ran{T\pm iI} ^\perp.
\end{equation} 
Then, the corresponding deficiency indices are defined as
\begin{equation}
\label{e-defectindex}
    n _\pm (T) := \mr{dim} \left( \mc{D} _\pm \right).
\end{equation} In this paper, we are only considering operators such that  $n_+ = n_- < \infty$.

By the functional calculus, if $T$ is self-adjoint (\emph{i.e.} $n_+ = n_- =0$) then 
$$ \bga (T) := (T-iI) (T+iI) ^{-1}, $$ is unitary. Conversely, given any unitary $U$ so that $1 \notin \sigma _p (U)$, $T:= \bga ^{-1} (U)$ is a (potentially unbounded) closed self-adjoint operator. This correspondence extends to symmetric operators and partial isometries. Namely if $T$ is a closed symmetric operator with defect indices $(n_+ , n_-)$, then 
$ \bga (T) $, restricted to\\ $\ran{T+iI}$ is an isometry onto $\ran{T-iI}$. Extending $\bga (T) $ by $0$ to an operator on the entire space $\H$, $\bga (T)$ is then a partial isometry, $V = \bga (T)$ with initial and range spaces:
$$ \ker{\bga (T)} ^\perp = \ran{T+iI}, \quad \ran{\bga (T)} = \ran{T-iI}. $$ The kernel and co-kernel of $V$ are then
$$ \ker{V} = \mc{D} _+, \quad \ran{V} ^\perp = \mc{D} _-, $$ so that the defect indices of $V$, 
$$ n_+ := \dim{\ker{V}}, \quad n_- := \dim{\ran{V} ^\perp}, $$ are the same as the deficiency indices of $T$. 

If $T$ has equal defect indices $(n,n)$, then it has a $\scr{U}_n -$parameter family of self-adjoint extensions \cite[Chapter 3]{AalbeverioKurasov2000}. Moreover, these can all be constructed explicitly as follows. First consider the Cayley Transform $V = \bga (T)$. Since $\dim{\ker{V}} = \dim{\ran{V}} = n$, for any $U \in \scr{U}_n$, we can define an isometry $\hat{U}: \ker{V} \rightarrow \ran{V} ^\perp$. Namely, fix an orthonormal basis for $\ker{V}$, $\{ \phi ^+ _k \} _{k=1} ^n$, and consider the isometry 
$$ W e_k = \phi ^+ _k, \quad W : \C ^n \rightarrow \ker{V} , $$ where $\{e_k \}$ is the standard basis of $\C ^n$. 
Similarly, if $\{ \phi ^- _k \}$ is a fixed ON basis of $\ran{V} ^\perp$, we define
$$ W' e_k = \phi ^- _k. $$ The isometry $\hat{U} : \ker{V} \rightarrow \ran{V} ^\perp$ is then given by 
$$ \hat{U} := W' U W^*. $$ It follows that all possible unitary extensions of $V = \bga (T)$ are obtained as:
$$ V _U := V + \hat{U}. $$ Indeed, we have:
\begin{figure}[h]
\adjustbox{scale=1, center}{\begin{tikzcd}
\H\arrow[d, "V_{U}" ] & = & \ker{V}^{\perp}\arrow[d, "V"] & \oplus &  \ker{V}\arrow[d, "\hat{U}"] \\
\H& = & \ran{V} & \oplus &  \ran{V}^\perp.
\end{tikzcd}}
   \label{f-KerRanV}
\end{figure}

One can prove that if $T$ is densely-defined and symmetric, then necessarily $1$ is not an eigenvalue of any unitary extension of $V = \bga (T)$, so that the set of all closed self-adjoint extensions of $T = T_U$ are then given by 
$$ T_U = \bga ^{-1} (V_U ). $$ Moreover, one can check that the domain of each $T_U$ is then given by
\[\dom{T_U} = \dom{T} + \ran{(\hat{U}-I) P _{+} },\]
where $P_+$ projects onto the defect space $\mc{D} _+ = \ker{T^* - iI}$.

\subsection{Operators of Interest}\label{ss-interest}
The primary examples that we will be investigating are powers of the derivative, both on finite intervals and the positive half-line.
\begin{defn}
\label{d-Kn}
Let $n \in \N.$ Then, we define the family of operators $K_n$ on $L^2(0, \infty)$ by
\begin{equation}
\label{e-KnDef}
K_n := ( -1 )^{n} \frac{d^{2n}}{dx^{2n}}.
\end{equation}
\end{defn}

For all $n \in \N$, the associated maximal operator $K_{n, \text{max}}$ are defined as $K_{n, \text{max}}f=K_n f$ for $f$ in the following domain
\begin{align}\label{e-DMaxKn}
\dom{K_{n, \text{max}}}
&:=\{
f\in L^2(0,\infty):
f^{(j)}\in \text{AC}\ti{loc}(0, \infty) \text{ for }j=0,\hdots, 2n-1
\nonumber \\
&
\qquad\qquad\qquad\qquad\qquad\qquad\qquad\text{ and }K_{n} f\in L^2(0, \infty)\}.
\end{align}
The minimal operator associated with $K_n$ is the adjoint of the maximal operator, $K_{n,\text{min}}:=K_{n,\text{max}}^*$, and are given by $K_{n, \text{min}}f=K_n f$ for $f$ in the domains
\small
\begin{align}\label{e-DMinKn}
\dom{K_{n, \text{min}}} &:= \left \{
f\in \dom{K_{n, \text{max}}}:
\lim_{x\to 0^{+}} f^{(j)}(x) = 0 \text{ for }j=0,\hdots, 2n-1 \right \}.
\end{align}
\normalsize
Note that we have that $K_n$ is regular at $x=0$, so the limits in~\eqref{e-DMinKn} exist finitely and we can define point evaluation at $x=0$ for $f \in \dom{K_{n, \text{max}}}$ and its derivatives up to order $2n-1$, via the limits
\[f^{(j)}(0) := \lim_{x\to 0^{+}} f^{(j)}(x) \quad \text{for } j =0, \dots, 2n-1.\]

The second family of operators that we will examine is constructed in a manner similar to the first. 
\begin{defn}
\label{d-Ln}
Let $a>0$ be a real number and $n \in \N$. Then, on $L^2(-a,a)$ we define $L_n$ to be the family of differential operators associated with the expression
\begin{equation}
\label{e-LnDef}
    L_n := i^n \frac{d^n}{dx^n}.
\end{equation}
\end{defn}
Here we note that one can define $L_n$ on a non-symmetric interval $(a,b)$ and achieve similar results to the ones that we present later in this paper. For the sake of computational simplicity, we have chosen to present the case where the interval is symmetric about the origin.   

The maximal and minimal domains associated with $L^n$  are defined analogously to~\eqref{e-DMaxKn} and~\eqref{e-DMinKn}, respectively,
\begin{align*}
\dom{L_{n, \text{max}}}&:=  \{
f\in L^2(-a,a):
f^{(j)}\in \text{AC}\ti{loc}(-a,a) \text{ for }j=0,\hdots, n-1 
\nonumber \\
&
\qquad\qquad\qquad\qquad\qquad\qquad\qquad  \text{ and }L_{n} f\in L^2(-a,a) \}
\end{align*}
and
\small
\begin{align}\label{e-DMinLn}
\dom{L_{n, \text{min}}} &:= \left \{
f\in \dom{L_{n, \text{max}}}:
\lim_{x\to \pm a} f^{(j)}(x) = 0 \text{ for }j=0,\hdots, n-1 \right \}.
\end{align}
\normalsize
Since $L_n$ is regular at $x=\pm a$, limits of $f \in \dom{L_{n, \text{max}}}$ (and its derivatives up to order $n-1$) as $x$ approaches $\pm a$ are well-defined and exist finitely. Thus, it makes sense to define the minimal domain in terms of limits as in~\eqref{e-DMinLn}. Furthermore, the fact that these limits exist finitely gives us the ability to talk about point evaluation of $f$ and its derivatives up to order $n-1$ at the endpoints of the interval $(-a,a)$, \[f^{(j)}(\pm a) := \lim_{x\to \pm a} f^{(j)}(x) \quad \text{for } j =0, \dots, n-1.\]

The families of operators $L_{n, \text{min}}$ and $K_{n, \text{min}}$ are all closed, densely-defined, and symmetric linear operators defined on subsets of the Hilbert spaces that they act on. Thus, we can discuss their defect spaces and indices as defined in~\eqref{e-defect} and~\eqref{e-defectindex}. For all $n$, we have 
\[n _\pm (L_{n, \text{min}}) = n _\pm (K_{n, \text{min}}) = n.\]
In general, the operators $L_{n, \text{min}}$ and $K_{n, \text{min}}$ can have more than $n$ linearly independent eigenfunctions. However, in both cases, it will turn out that only $n$ of them will belong to the appropriate $L^2$ space. Thus, describing bases for their defect spaces necessitates some care.

It is important to note that throughout this paper any roots of complex numbers are to be interpreted as the principal branch; appropriate multiplicative factors will be included to rotate a principal root to a non-principal root if needed.

\subsection{Self-Adjoint Extensions for Rank-One Perturbations}
We now present the classical framework used to describe self-adjoint extensions for operators with deficiency indices $(1,1)$ (i.e. rank-one perturbations). First, we describe the extensions in terms of rank-one perturbations, then in terms of boundary conditions, and end by connecting the two descriptions. These results will be generalized to rank-two perturbations and finite-rank perturbations in Sections~\ref{s-RankTwoExtensions} and~\ref{s-FiniteRankExtensions}, respectively.

First, we examine $K_1 = - \frac{d^2}{dx^2}$. As previously mentioned, $K_1$ is defined on a suitable dense domain and is a symmetric, closed, with defect indices $(1,1)$. So all self-adjoint extensions of $K_{1,\text{min}}$ are indexed by a complex parameter $\alpha \in \scr{U}_1$. 
The defect spaces are one-dimensional:
$$ \mc{D} _+ = \bigvee \phi ^+ , \quad \phi ^+ (x) = \sqrt[4]{2}e^{\sqrt{-i}x}, $$ and 
$$ \mc{D} _- = \bigvee \phi ^-, \quad \phi ^- (x) = \sqrt[4]{2}e^{-\sqrt{i} x}. $$ (Here we have normalized the defect vectors.) It follows that all unitary extensions of $V = \gamma (K_1)$ are given by:
$$ U_\alpha := V + \alpha \ip{\phi ^+}{\cdot} \phi^-, \quad \quad \alpha \in \scr{U}_1, $$ and all self-adjoint 
extensions of $K_1$ are given by $K_{1,\alpha} = \gamma ^{-1} (U_\alpha)$, where 
\be \label{domalphaK} \dom{K_{1,\alpha}} = \dom{K_{1,\text{min}}} + \bigvee ( \alpha \phi ^- - \phi ^+ ). \ee
On the other hand, all self-adjoint extensions of $K_{1,\text{min}}$  can be obtained by imposing boundary conditions on the domain of $K_{1,\text{max}}$ for $b,c \in \scr{U}_1$:  $$ \dom{K_1(b,c)} := \{ f \in \dom{K^*_{1}} | \ b f(0) + c f^{\prime} (0) = 0, \ \mbox{and} \ b \ov{c} \in \R \}. $$ 
\begin{prop}
\label{DBetaSAK}
Each operator $K_1(b,c)$, for $b,c \in \scr{U}_1$  is self-adjoint.
\end{prop}
\begin{proof}
Recall that every element $f \in \dom{K_{1,\text{min}}}$ is absolutely continuous and obeys $f(0) = 0 $. Hence $K_1(b,c)$ is a symmetric extension of $K_{1,\text{min}}$. On the other hand it follows that $K_1(b,c)$ is a symmetric restriction of $K_{1,\text{max}}$. Also, since $K_1(b,c)$ is a symmetric operator by the previous claim, we have the following containment of domains,
\[\dom{K_{1,\text{min}}} \subset \dom{K_1(b,c)}  \subset \dom{K_1(b,c)^{*}}  \subset \dom{K_{1,\text{max}}}. \]
Now suppose that $g \in \dom{K_1(b,c)^{*}}$. Then for any $f \in \dom{K_1(b,c)}$, calculate that
\begin{align*}
    \ip{K_1(b,c)f}{g} &=  \int_{0}^{\infty} -f^{\prime \prime} (t) \ov{g} (t) dt\\
    &= -f^{\prime} (t) \ov{g} (t) \vert_{t=0}^{\infty} + \int_{0}^{\infty} f^{\prime} (t) \ov{g}^{\prime} (t) dt\\
    &= \left. \left ( -f^{\prime} (t) \ov{g} (t) + f(t) \ov{g}^{\prime}  (t) \right) \right \vert_{t=0}^{\infty}+ \int_{0}^{\infty} f(t) \left ( -\ov{g}^{\prime \prime} (t) \right ) dt\\
    &= f^{\prime} (0) \ov{g} (0) - f(0) \ov{g}^{\prime}  (0)  + \ip{f}{K_1(b,c)^{*}g}\\
    &= -f(0) \left ( \ov{c}b  \ov{g} (0) + \ov{g}^{\prime}  (0) \right ) + \ip{f}{K_1(b,c)^{*}g}.
\end{align*}
From the definition of the adjoint,
this proves that 
$$ 0 = \ov{c}b  \ov{g} (0) + \ov{g}^{\prime}  (0), $$ or by taking complex conjugates and using the fact that $b \ov{c}  \in \R$ and $\vert c \vert =1$,
$$  b g(0) + c g^{\prime} (0) = 0.$$
Hence, $g \in \dom{K_1(b,c)}$. Therefore, $\dom{K_1(b,c)} = \dom{K_1(b,c)}^{*}$, as desired.
\end{proof}
\begin{remark}
Note that attempts to  generalize this technique--of integrating by parts, applying the adjoint definition, then solving an equation to obtain boundary conditions--fail to yield solvable expressions for higher rank problems. Therefore, the methods presented in Section \ref{s-FiniteRankExtensions} below are not direct generalizations of the technique used in this section.
\end{remark}
\begin{prop}
\label{p-PertBCR1HalfLine}
The extension domain described via a rank-one perturbation $\dom{K_{1,\alpha}}$ and the extension domain described by boundary conditions \\ $\dom{K_1 (b,c)}$ coincide, i.e. $\dom{K_{1,\alpha}} = \dom{K_1 (b,c)}$, when the perturbation parameter $\alpha$ and the boundary condition parameters $b,c$ satisfy
\[\alpha = \frac{b + c\sqrt{-i}}{b -c\sqrt{i}}.\]
\end{prop}
\begin{proof}
From Equation~\eqref{domalphaK}, $f \in \dom{K_{1,\alpha}}$, then \begin{align*}
    f(0) &= d \left ( \alpha e^{-\sqrt{i} 0} - e^{\sqrt{-i} 0} \right ) = d(\alpha-1),\\
    f^{\prime}(0) &= d \left ( -\sqrt{i}\alpha e^{-\sqrt{i} 0} - \sqrt{-i}e^{\sqrt{-i} 0} \right ) = d(-\sqrt{i}\alpha-\sqrt{-i}),
\end{align*}
for a constant $d \in \C$. Substituting the above into $b f(0) + c f^{\prime} (0) = 0$ and solving for $\alpha$ yields the result.
\end{proof}

Now, we will look at the rank-one case on the interval $L_1 = i \frac{d}{dx}$. We again have that this operator is defined on a suitable dense domain and is a symmetric, closed, with defect indices $(1,1)$. Thus, its defect spaces are one-dimensional and they are given by
$$ \mc{D} _+ = \bigvee \phi ^+ , \quad \phi ^+ (x) = \frac{e^x}{\sqrt{2a}}, $$ and 
$$ \mc{D} _- = \bigvee \phi ^-, \quad \phi ^- (x) = \frac{e^{-x}}{\sqrt{2a}}. $$ 

Thus, the unitary extensions of $V = \gamma (L_1)$ are given by: $$ U_\alpha := V + \alpha \ip{\phi ^+}{\cdot} \phi^-, \quad \quad \alpha \in \scr{U}_1, $$ and all self-adjoint extensions of $L_{1, \text{min}}$ are given by $L_{1,\alpha} = \gamma ^{-1} (U_\alpha)$, where 
\[\dom{L_{1,\alpha}} = \dom{L_{1, \text{min}}} + \bigvee ( \alpha \phi ^- - \phi ^+ ).\]
We can also express the self-adjoint extensions of $L_{1, \text{min}}$ can be obtained by placing boundary conditions on the domain of $L_{1, \text{max}}$. For any $\beta \in \scr{U}_1$, define 
$$ \dom{L_1(\beta )} := \{ f \in \dom{L_{1, \text{max}}} | \ f(a) = \beta f(-a) \}. $$
Then, we have the following results regarding self-adjoint extensions,
\begin{prop}
\label{DBetaSA}
Each operator $L_1(\beta )$, for $\beta \in \scr{U}_1$ is self-adjoint.
\end{prop}
\begin{prop}
\label{p-PertBCR1Interval}
The extension domain described via a rank-one perturbation $\dom{L_{1,\alpha}}$ and the extension domain described by boundary conditions \\ $\dom{L_1 (\beta)}$ coincide, i.e. $\dom{L_{1,\alpha}} = \dom{L_1 (\beta )}$, when the perturbation parameter $\alpha$ and the boundary condition parameter $\beta$ satisfy
\[\alpha = \frac{\beta e^{-2a}-1}{\beta - e^{-2a}}.\]
\end{prop}
The proofs of Propositions~\ref{DBetaSA}  and \ref{p-PertBCR1Interval} are omitted because they are analogous to the proofs of Propositions~\ref{DBetaSAK} and \ref{p-PertBCR1HalfLine}, respectively.

\subsection{Self-Adjoint Boundary Conditions for Rank-Two Perturbations} \label{s-RankTwoExtensions}
In order to establish a family of boundary conditions with which we capture all self-adjoint extensions of $L_2$, we will utilize the methods detailed by Eckhardt et al.~in \cite{EGNT2013}. Their work heavily relies on the explicit inversion formula of $2\times 2$ matrices and, so, their proofs do not immediately generalize to the higher rank setting. Yet, this section forms a warm up for some of the ideas used in Section~\ref{s-FiniteRankExtensions}, in which the finite-rank case is carried out.
\begin{defn}
An operator $\tau$ is a {\bf{Sturm--Liouville} differential operator} if it has the form  \begin{equation} \label{sturmliouville}
    \tau f := \frac{1}{r} \left ( - \left ( p (f^{\prime}+sf)\right )^{\prime} +sp(f^{\prime}+sf) +qf \right ),
\end{equation} where $p,q,s,r$ are real-valued Lebesgue measurable functions on an interval $(a,b)$ with the properties
\begin{enumerate}
    \item $p \neq 0$ a.e. on $(a,b)$,
    \item $r > 0$ a.e. on $(a,b)$,
    \item $p^{-1},q,r,s \in L_{\text{loc}}^{1} ((a,b))$, 
\end{enumerate} and has domain \[ \dom{\tau}= \left \{ f \in AC_{\text{loc}} ((a,b)) : p (f^{\prime}+sf) \in AC_{\text{loc}} ((a,b)) \right \}. \]
\end{defn}

\begin{defn}
Given a function $f \in AC_{\text{loc}} ((a,b))$ and the coefficent functions $p$ and $s$ of a Sturm--Liouville differential operator as described above, we define the {\bf{first quasi derivative}} $f^{[1]}$ of $f$ by $f^{[1]} := p (f^{\prime}+sf)$.
\end{defn}
When analyzing Sturm--Liouville operators, utilizing the first quasi derivative in places where one might use the classical first derivative in non-Sturm--Liouville theory is quite standard and has lead to myriad results. In the following argument, this is done to modify the Wronskian determinant. For $f,g \in \dom{\tau}$, define the {\bf{modified Wronskian determinant}} $W(f,g) (x)$ as
\[ W(f,g) (x) = f(x)g^{[1]} (x)- f^{[1]} (x) g(x).\] Throughout the rest of this paper, we will refer to the modified Wronskian determinant simply as the Wronskian. 

Next, we will discuss the maximal and minimal operators associated with the Sturm--Liouville differential operator on  $L^2 ((a,b);r(x)dx)$. The associated maximal operator $T_{\text{max}}$ is given by $T_{\text{max}} f = \tau f$ for $f$ in the maximal domain 
\small
\[ \dom{T_{\text{max}}}= \left \{ f \in L^2 ((a,b);r(x)dx): f \in \dom{\tau}, \tau f \in L^2 ((a,b);r(x)dx) \right \}. \] 
\normalsize
The associated minimal operator $T_{\text{min}}$  is defined as $T_{\text{min}}f = \tau f$ for $f$ in the minimal domain 
\small
\[ \dom{T_{\text{min}}}= \left \{  f \in \dom{T_{\text{max}}} : \forall g \in  \dom{T_{\text{max}}}, W(f,g)(a)=W(f,g)(b)=0 \right \}. \]
\normalsize
 Before we discuss self-adjoint extensions/restrictions of these operators, we first examine the behavior that these operators can exhibit near the boundary points, as that behavior determines the deficiency indices. 
\begin{defn}
\label{d-RegularRankTwo}
A Sturm--Liouville operator is {\bf{regular}} at an endpoint (either $a$ or $b$) if the coefficient functions $p^{-1},q,r,s$ are integrable near that endpoint. 
\end{defn}
\begin{defn}
\label{d-LcLpRankTwo}
A Sturm--Liouville operator $\tau$ is said to be in the {\bf{}limit-circle} (l.c.) case at an endpoint if for all $z \in \C$, all solutions to the eigenvalue equation $(\tau - z)u=0$ are in $L^2 ((a,b);r(x)dx)$ near that endpoint. Similarly, we say that $\tau$ is in the {\bf{}limit-point} (l.p.) case at an endpoint if for all $z \in \C$, there exists a solution to the eigenvalue equation $(\tau - z)u=0$ that is not in $L^2 ((a,b);r(x)dx)$ near that endpoint.
\end{defn}
 Note that $\tau$ being regular near at endpoint implies that $\tau$ is l.c. at that endpoint (see, e.g.~\cite[Lemma 4.2]{EGNT2013}). Further note that regularity at an endpoint allows for the evaluation of functions in the maximal domain at that endpoint via taking limits (from within the domain).
\begin{thm}[{see, e.g.~\cite[Theorem 4.6]{EGNT2013}}]
The deficiency indices of $T_{\text{min}}$ are as follows
\begin{itemize}
    \item $(0,0)$ if $\tau$ is l.c. at neither $a$ nor $b$,
    \item $(1,1)$ if $\tau$ is l.c. at either $a$ or $b$, but not l.c. at both,
    \item  $(2,2)$ if $\tau$ is l.c. at both $a$ and $b$.
\end{itemize}
\end{thm}

Moving forward, we will be concerned with the case where $\tau$ is l.c. at both $a$ and $b$. Since we are dealing with finite deficiency indices, self-adjoint restrictions of $T_{\text{max}}$ and self-adjoint extensions of $T_{\text{min}}$ have a one-to-one correspondence, so we will use the terms interchangeably. To describe the self-adjoint restrictions of $T_{\text{max}}$, we first  let $w_1, w_2 \in \dom{T_{\text{max}}}$ be two functions with the following Wronskians
\begin{itemize}
    \item $W(w_1, \ov{w_2} ) (a)=  W(w_1, \ov{w_2} ) (b)=1$,
    \item $W(w_1, \ov{w_1} ) (a)=  W(w_1, \ov{w_1} ) (b)=0$,
    \item $W(w_2, \ov{w_2} ) (a)=  W(w_2, \ov{w_2} ) (b)=0$.
\end{itemize}
Towards the goal of determining self-adjoint boundary conditions, we use $w_1$ and $w_2$ to define the following linear functionals \[ BC_{\star}^{1} (f):= W (f,\ov{w_2} ) (\star),   BC_{\star}^{2} (f):= W (\ov{w_1}, f) (\star),\]for $f \in \dom{T_{\text{max}}}$ and $\star= a \text{ or } b$. Then, we have the following result for these functionals dependent on the behavior of $\tau$ at the $a$ and $b$.

\begin{lemma}[{see, e.g.~\cite[Lemma 6.1]{EGNT2013}}] \label{L6.1EGNT2013}
Let $\tau$ be regular at $\star$. Then, there exists $w_1, w_2 \in \dom{T_{\text{max}}}$ such that \[ BC_{\star}^{1} (f)=f(
\star) \text{ and }   BC_{\star}^{2} (f)=f^{[1]} (
\star) \text{ for all } f \in \dom{T_{\text{max}}},\]
where we again have $\star= a \text{ or } b$.
\end{lemma}

\begin{thm}[{see, e.g.~\cite[Theorem 6.3]{EGNT2013}}]
\label{T6.3EGNT2013}
Let $\tau$ be l.c.~at both endpoints. Then, an operator $S$ is a self-adjoint restriction of $T_{max}$ if and only if there exists matrices $B_a, B_b \in \C^{2 \times 2}$ with the properties \[ \text{rank} (B_a \vert B_b )=2 \text{ and } B_a J B_a^{*} = B_b J B_b^{*} \text{ with } J= \begin{pmatrix}
0 & -1 \\
1 & 0
\end{pmatrix}, \] such that \[ \dom{S}= \left \{ f \in \dom{T_{\text{max}}}: B_{a} \begin{pmatrix}[1.2]
BC_{a}^{1} (f)\\
BC_{a}^{2} (f)
\end{pmatrix} = B_{b} \begin{pmatrix}[1.2]
BC_{b}^{1} (f)\\
BC_{b}^{2} (f)
\end{pmatrix} \right \}.\]
\end{thm}

\section{Matrix-Valued Clark Theory}
\label{s-matrixclark}
In general, spectral measures can be decomposed into three parts: the absolutely continuous (with respect to Lebesgue measure), pure-point, and singular continuous spectra. Via known results for Clark measures on the disk~\cite{LMT}, we present 
formulae that describe the absolutely continuous and pure-point parts,
as they will be used later in our examples to compute spectral measures for derivatives operators. In~\cite{LMT}, it was also discussed that the description of the singular continuous parts is often not explicitly accessible.  

At this time, we note that the uniqueness of the spectral measure for a given self-adjoint cyclic operator $T$ is only guaranteed by choosing a basis of a cyclic vector $\varphi$. By choosing a different cyclic vector, one obtains a different (though unitarily equivalent) spectral measure for $T$. For the matrix-valued setting, we show analog statements of equivalence in Theorems~\ref{t-LivsicEquivAC} and~\ref{t-LivsicEquivPP}. This will allow us to fix a particular cyclic subspace of $\H$ for each operator we examine (via choosing bases for its defect spaces). As such, there will be a unique (Clark) spectral measure in the contexts we are working with. Because of this uniqueness, we will refer to them as ``the spectral measure" of the operator.

\subsection{Liv{\v s}ic characteristic function}
\label{ss-AMR}
In order to use Clark Theory to investigate the spectral properties of differential operators, we must first construct an analytic function that encodes the spectral information of those operators. To that end, we will employ a Theorem of Liv{\v s}ic  \cite{LivsicIso, LivsicLinHil} using a construction outlined by Aleman--Martin--Ross~\cite{AMR}. We first let $T$ be a simple, symmetric, closed operator with finite deficiency indices $(n,n)$ that is densely defined on a Hilbert space $\mathcal{H}$. Note that since all of our operators of interest are differential operators, the underlying Hilbert spaces will be suitable function spaces. Given a basis $\left \{ \psi_{k} \right \}_{k=1}^{n}$ and an orthonormal basis $\left \{ \tilde{\psi}_{k} \right \}_{k=1}^{n}$ for the positive defect space of $T$, we construct the matrix-valued function
\begin{equation} \label{EQdefA}
A_{T} (w,z) = \begin{pmatrix}
\ip{\psi_{1} (w;x)}{\tilde{\psi}_1 (z;x)}_{\mathcal{H}} & \dots & \ip{\psi_{1} (w;x)}{\tilde{\psi}_n (z;x)}_{\mathcal{H}}
\\
\vdots&\ddots&\vdots\\
\ip{\psi_{n} (w;x)}{\tilde{\psi}_1 (z;x)}_{\mathcal{H}} & \dots & \ip{\psi_{n} (w;x)}{\tilde{\psi}_n (z;x)}_{\mathcal{H}}
\\
\end{pmatrix}
\end{equation}
for $w \in \C_{+}$ and $z=\pm i$, where $x$ is the independent variable of the functions in $\mathcal{H}$ and  $\ip{\cdot}{\cdot}_{\mathcal{H}}$ denotes the inner product on $\mathcal{H}$. From this, we then define the matrix-valued function on $\C_{+}$,
\begin{equation}\label{EQdefB}
B_{T} (w) = \frac{w-i}{w+i}(A_{T}(w, i))^{-1} A_{T}(w, -i).
\end{equation}
The function $A_{T}$ is known as {\bf{analytic Gram matrix}} of T and $B_{T}$ is the {\bf{Liv{\v s}ic characteristic function}} of T. If $\left \{ \psi_{k}^{1} \right \}_{k=1}^{n}$ and $\left \{ \psi_{k}^{2} \right \}_{k=1}^{n}$ are different bases for the positive defect space of $T$ and we construct Liv{\v s}ic characteristic functions $B_T^{1}$ and $B_T^{2}$ from them, then there exist constant matrices  $R, Q \in \scr{U}_n$ such that for all $w \in \C_{+}$,
\begin{equation}
\label{e-LivsicEquiv}
    B_T^{1} (w) = R B_T^{2} (w) Q.
\end{equation}
More broadly speaking, we say that two Liv{\v s}ic characteristic functions $B,C$ are equivalent if the condition~\eqref{e-LivsicEquiv} holds; i.e. if there exists constant matrices  $R, Q \in \scr{U}_n$ such that $B(w)=RC(w)Q$ for all $w \in \C_{+}$. With this notion, we now discuss the aforementioned theorem of Liv{\v s}ic.
\begin{thm}[{see, e.g.~\cite{LivsicIso, LivsicLinHil}}]
\label{t-Livsic}
Let $\mathcal{H}_1 , \mathcal{H}_2$ be Hilbert spaces. Further let $T_1$ and $T_2$ be simple, symmetric, closed, expressions with finite deficiency indices $(n,n)$ that are densely defined on $\mathcal{H}_1$ and $\mathcal{H}_2$, respectively. Then, $T_1$ and $T_2$ are unitarily equivalent if and only if $B_{T_1}$ and $B_{T_2}$ are equivalent in the sense of Equation~\eqref{e-LivsicEquiv}.
\end{thm}

As is detailed in \cite{AMR}, there is a close connection between the Liv{\v s}ic characteristic function and reproducing kernel Hilbert spaces (RKHS in short) of analytic functions. Namely, a simple, symmetric, closed, densely defined expression $T$ with deficiency indices $(n,n)$ can be associated with a vector-valued RKHS of analytic functions on $\C \setminus \R$ via a model space. The kernel of that RKHS can then be used to produce a meromorphic function that is equivalent to the Liv{\v s}ic characteristic function in the sense of Equation~\eqref{e-LivsicEquiv}. In~\cite{LMT}, it is further shown that the Clark measures of that Liv{\v s}ic characteristic function will then be spectral measures of the self-adjoint extensions of the model operator associated with the expression $T$. 

Moving forward, the specific expression $T$ that an analytic Gram matrix or a Liv{\v s}ic characteristic function is associated with will either be clear from the context or irrelevant to what is being discussed in that context. As such, we drop the subscript $T$.

\subsection{Absolutely continuous spectrum}
The first piece of the Lebesgue decomposition of Clark measures that we will look at is the absolutely continuous part of the measure, which entails computing its Lebesgue weight (also known as its Radon--Nikodym derivative). We begin by noting the following result for the absolutely continuous part of Clark measures corresponding to Liv{\v s}ic characteristic functions on the disk \cite[Theorem 2.1]{LMT}.
\begin{lemma}
\label{diskAC}
For $b \in \scr{S} _n (\D )$ the Lebesgue weight of the absolutely continuous part of the Clark measure $\bmu^{b \alpha^{*}}$, denoted by $\frac{d \bmu^{b \alpha^{*}}}{d \mathfrak{m}}$ is given by \begin{equation}
    \label{e-mtxACpartb}
    \frac{d \bmu^{b \alpha^{*}}}{d \mathfrak{m}} (\lambda) = \lim _{\zeta \nt \la} \left( I- \alpha b^{*} (\zeta) \right)^{-1} \Delta^{2}_{b \alpha^{*}} (\zeta) \left( I- b (\zeta) \alpha^{*} \right)^{-1}; 
\end{equation}  
for $\zeta \in \D$ and  $\mathfrak{m}\text{--a.e.~} \lambda \in \T,$ where $\mathfrak{m}$ is the Lebesgue measure on $\T$ and $\Delta_{b \alpha^{*}} (\zeta) = ( I- \alpha b^{*} (\zeta) b (\zeta) \alpha^{*} )^{1/2}$.
\end{lemma} 
Note that in general, the limits used to describe the absolutely continuous spectrum are defined Lebesgue almost everywhere. To simplify notation, we will not repeat this fact in any future instances where such a limit is defined or applied.
\begin{remark} Equation~\eqref{e-mtxACpartb} is the matrix-valued analog for a classical scalar-valued Clark theory result, that the spectral measure has Lebesgue density \[\frac{d \mu}{d \mathfrak{m}} (\lambda) = \lim _{\zeta \nt \la} \frac{1-\left \vert b (\zeta) \right \vert^{2}}{\left \vert \alpha - b (\zeta) \right \vert^{2}}; \quad \lambda \in \T, \zeta \in \D,\]
where $b$ is scalar-valued~\cite[Proposition 9.1.14]{RossCT}.
\end{remark}
We translate this to the upper half-plane. Specifically, $B(w)=b(\gamma(w))$, where $\ga (w) = (w-i)(w+i)^{-1}$ as was defined in Section \ref{ss-notation}. 
\begin{lemma}
\label{disktoplaneAC}
For $B\in \scr{S} _n (\C ^+ )$, the Lebesgue weight of the absolutely continuous part of the Clark measure $\bmu^{B \alpha^{*}}$ is given by 
\[\frac{d \bmu^{B \alpha^{*}}}{dm} (s) = \frac{1}{\pi(1+s^2)} \frac{d \bmu^{b \alpha^{*}}}{d\mathfrak{m}} (\gamma (s)); \quad s \in \R,\]  
where $m$ is the Lebesgue measure on $\R$.
\end{lemma} 
Using Lemmas~\ref{diskAC} and \ref{disktoplaneAC}, we now compute the density of the absolutely continuous part of the Clark measure $\bmu^{B \alpha^{*}}.$
\begin{thm}
\label{mtxAC}
For $B\in \scr{S} _n (\C ^+ )$, the Lebesgue weight of the absolutely continuous part of the Clark measure $\bmu^{B \alpha^{*}}$ is given by
\small
\[\frac{d \bmu^{B \alpha^{*} }}{dm} (s)= \frac{1}{\pi(1+s^2)} \lim _{w \downarrow s}  (\alpha^{*}-B^{*}(w))^{-1} (I-B^{*}(w)B(w)) (\alpha-B(w))^{-1},\]
\normalsize
for $s \in \R$ and $ w \in \C_+$.
\end{thm}

\begin{proof}
We begin with a substitution of Lemma~\ref{diskAC} into Lemma~\ref{disktoplaneAC},

\small
\[\frac{d \bmu^{B \alpha^{*}}}{dm} (s)= \frac{1}{\pi(1+s^2)} \lim _{\gamma (w) \nt \gamma (s)} \left( I- \alpha b^{*} (\gamma (w)) \right)^{-1} \Delta^{2}_{b \alpha^{*}} (\gamma (w)) \left( I- b (\gamma (w)) \alpha^{*} \right)^{-1}, \]
\normalsize where we have identified $z \in \D$ as the Cayley transform of $w \in \C_{+}$ and $\lambda \in \T$ as the Cayley transform of $s \in \R$. Next, we compute
\begin{align*}
    \frac{d \bmu^{B \alpha^{*}}}{dm} (s) &= \frac{1}{\pi(1+s^2)} \lim _{w \downarrow s} \left( I- \alpha B^{*}(w) \right)^{-1} \Delta^{2}_{B \alpha^{*}} (w) \left( I- B (w) \alpha^{*} \right)^{-1}\\
    &=\frac{1}{\pi(1+s^2)} \lim _{w \downarrow s} \left( \alpha^{*}- B^{*}(w) \right)^{-1} \alpha^{*} \Delta^{2}_{B \alpha^{*}} (w) \alpha\left( \alpha- B (w) \right)^{-1}.
\end{align*}
Substituting \[\Delta^{2}_{B \alpha^{*}} (w) =( I- \alpha B^{*} (w) B (w) \alpha^{*} )=\alpha (I-B^{*} (w) B (w)) \alpha^{*}\] into the above yields the desired result.
\end{proof}

Next, we address the well-posedness of this construction with respect to the choice of basis. Recall that Liv{\v s}ic's Theorem tells us that choosing different bases for our defect spaces yields equivalent characteristic functions. Therefore, we examine how the absolutely continuous parts of Clark measures of equivalent characteristic functions are related. The following result verifies expectations and confirms the well-posedness for the absolutely continuous part. Later, we obtain an analogous result for the pure point parts.

\begin{thm}
\label{t-LivsicEquivAC}
Let functions $B_1 , B_2\in \scr{S} _n (\C ^+ )$ be equivalent such that $B_1 (w)= R B_2 (w)Q$ for all $w \in \C_+$, where $R,Q \in \scr{U}_n$ are fixed constant matrices. Then, the Lebesgue weight of the absolutely continuous parts of their associated Clark measures are unitarily equivalent with a perturbation parameter conjugated by those matrices of equivalency. Namely,
\[\frac{d \bmu^{B_1 \alpha^{*} }}{dm} (s)=R \frac{d \bmu^{B_2 (Q\alpha^{*}R)}}{dm} (s) R^{*},\]
for all $s \in \R$.
\end{thm}
\begin{proof} Suppose  $B_1 , B_2\in \scr{S} _n (\C ^+ )$ are equivalent as described in the hypothesis. Then,
\small
\begin{align*}
     &\pi(1+s^2) \frac{d \bmu^{B_1 \alpha^{*} }}{dm} (s)\\
     &= \lim _{w \downarrow s}  (\alpha^{*}-(R B_2 (w) Q)^{*})^{-1} (I-(R B_2 (w) Q)^{*}R B_2 (w) Q) (\alpha-R B_2 (w)Q)^{-1}\\
     &=\lim _{w \downarrow s}  R (Q\alpha^{*}R -B_{2}^{*} (w))^{-1} Q (I-Q^{*}B_{2}^{*} (w) B_2 (w) Q) Q^{*} (R^{*}\alpha Q^{*}-B_2 (w))^{-1} R^{*}\\
     &= R \left ( \lim _{w \downarrow s}  (Q\alpha^{*}R -B_{2}^{*} (w))^{-1} (I-B_{2}^{*} (w) B_2 (w) ) (R^{*}\alpha Q^{*}-B_2 (w))^{-1} \right ) R^{*}\\
     &= \pi(1+s^2) R \frac{d \bmu^{B_2 (Q\alpha^{*}R)}}{dm} (s) R^{*}.
\end{align*}
\normalsize
\end{proof}

\subsection{Nevanlinna's Formula}
In \cite[Theorem 3.1]{LMT}, some of the authors obtained a generalization of Nevanlinna's result on characterizing the pure point spectrum of a finite-rank unitary perturbation in terms of the Liv{\v s}ic characteristic function:

\begin{thm}\label{t-Nevanlinna}
Fix $b \in \scr{S}_n (\D )$ and $\alpha \in \scr{U}_n $. Then for any $\la \in \scr{U}_1$, 
$$ \bmu^{b \alpha^{*}} \left( \{ \la \} \right) = \lim _{\zeta \nt \la} (1-\zeta \bar\la) (I - b(z) \alpha ^* ) ^{-1}.$$ 
\end{thm}

For our purposes, we translate this result to the upper half-plane.
\begin{thm}\label{t-NevanlinnaB}
Fix $B \in \scr{S}_n (\C_{+} )$ and $\alpha \in \scr{U}_n $. Then for any $s \in \R$, 
\begin{equation}
    \label{uhpnev}
    \bmu^{B \alpha ^* } \left( \{ s \} \right) = \frac{2i}{\pi (1 +s^2 )^2} \lim _{w \downarrow s} (s-w) (I - B(w) \alpha ^* ) ^{-1}.
\end{equation}
\end{thm}
\begin{proof}
For $B\in \scr{S} _n (\C ^+ )$ and $\alpha \in \mc{U} _n$, we define the \emph{Nevanlinna function}:
$$ H_\alpha (w) := (I - B(w) \alpha ^* ) ^{-1} (I + B(w) \alpha ^* )$$ with $b$ as above, and $\zeta = \ga (w)$. Then
$$ h_\alpha (\zeta) = (I - b(\zeta) \alpha ^* ) ^{-1} (I + b(\zeta) \alpha ^* ) = H_\alpha (w); \quad \zeta = \ga (w), $$ is a Herglotz function on the disk. 
By the Nevanlinna represtentation formula for positive harmonic functions on $\C ^+$, there is a unique (Clark) measure $\bmu^{B\alpha ^*}$ so that
$$ \re{H_\alpha (w)} = c_\alpha \re{w} + \intfty \re{\frac{1}{i\pi} \frac{1}{t-w}} \bmu^{B\alpha ^*} (dt). $$ 
Furthermore, one can check that:
$$ c_\alpha = \bmu^{b\alpha ^*} \left( \{ 1 \} \right), $$ and that for any Borel set $\Om \subseteq \R$,
$$ \bmu^{B\alpha ^*} (\Om ) = \int _\Om \frac{1}{\pi (1+t^2)} (\bmu ^{b\alpha ^*}   \ga ) (dt). $$ In particular, given $\Om = \{ s \}$ for a fixed $s \in \R$,
\begin{align*}
    \bmu^{B \alpha ^* } \left( \{ s \} \right) &= \frac{1}{\pi (1 +s^2 )} \bmu ^{b\alpha ^*} \left( \{ \ga (s) \} \right)  \\
    &= \frac{1}{\pi (1 +s^2 )} \lim _{ \zeta \ \nt\ga (s) } (1 - \zeta \ga (s) ^* ) (I - b(\zeta) \alpha ^* ) ^{-1} \\
    &=\frac{1}{\pi (1 +s^2 )} \lim _{w \downarrow s} \left ( 1 - \frac{w-i}{w+i} \frac{s+i}{s-i} \right ) (I - B(w) \alpha ^* ) ^{-1}\\
    &=\frac{2i}{\pi (1 +s^2 )} \lim _{w \downarrow s} \frac{s-w}{(w+i)(s-i)} (I - B(w) \alpha ^* ) ^{-1}.\\
    &=\frac{2i}{\pi (1 +s^2 )^2} \lim _{w \downarrow s} (s-w) (I - B(w) \alpha ^* ) ^{-1}.
\end{align*}
\end{proof}

The following result confirms the well-posedness for the pure point part of the Clark measure.

\begin{thm}
\label{t-LivsicEquivPP}
Let functions $B_1 , B_2\in \scr{S}_n (\C_{+} )$ be equivalent such that $B_1 (w)= R B_2 (w)Q$ for all $w \in \C_+$, where $R,Q \in \scr{U}_n$ are fixed constant matrices. Then, the pure-point parts of their associated Clark measures are unitarily equivalent with a perturbation parameter conjugated by those matrices of equivalency. Specifically,
\[\bmu^{B_1 \alpha^{*} }(\{s\})=R  \bmu^{B_2 (Q\alpha^{*}R)} (\{s\}) R^{*},\]
for all $s \in \R$.
\end{thm}
\begin{proof} Suppose $B_1 , B_2\in \scr{S}_n (\C_{+} )$ are equivalent as described in the hypothesis. Then,
\begin{align*}
    \bmu^{B_1 \alpha ^* } \left( \{ s \} \right)
    &=\frac{2i}{\pi (1 +s^2 )^2} \lim _{w \downarrow s} (s-w) (I - R B_2(w) Q \alpha ^* ) ^{-1}\\
    &= \frac{2i}{\pi (1 +s^2 )^2} \lim _{w \downarrow s} (s-w) (R^{*} - B_2(w) Q \alpha ^* ) ^{-1} R^{*}\\
    &=R \left (\frac{2i}{\pi (1 +s^2 )^2} \lim _{w \downarrow s} (s-w) (I - B_2(w) Q \alpha ^* R) ^{-1} \right )R^{*}\\
    &= R \bmu^{B_2 (Q\alpha^{*}R) } \left( \{ s \} \right) R^{*},
\end{align*}
as desired.
\end{proof}
Here we note that what Theorem~\ref{t-LivsicEquivPP} states for the pure-point parts of Clark Measures is analogous to what Theorem~\ref{t-LivsicEquivAC} provided for the absolutely continuous parts. In order to discuss the implications of Clark measures being unitarily equivalent (with a transformation of the perturbation parameter) established, we will need results from Section~\ref{s-FiniteRankExtensions}. We will return to both of these theorems in Subsection~\ref{ss-FRCharacteristic}.

\section{Self-Adjoint Boundary Conditions for Finite-Rank Perturbations}
\label{s-FiniteRankExtensions}
In this section, we will apply the main theorem established by Sun et al.~in \cite{BSHZ2019} to express the set of all self-adjoint extensions of operators from a class that includes $L_n$ for $n$ even and $K_n$ for all $n$ in terms of boundary conditions. Then, we establish the correspondence between parameterizing self-adjoint extensions by boundary conditions and parameterizing self-adjoint extensions by the finite-rank perturbation parameter $\alpha$ in Theorems~\ref{T_FRAlphaBCCOnnection2} and \ref{T_FRAlphaBCCOnnection}.

First, we will discuss self-adjoint extensions in terms of finite-rank perturbations. Let $\phi_{k}^{\pm}$ for $k = 1, \dots, n$ be the basis functions for the defect spaces $\mathcal{D}_{\pm}$ of an operator $T$ as defined in Equation~\eqref{e-defect}. \begin{defn}
\label{d-genalpha} 
The domains of the self-adjoint extensions of an operator $T$ with deficiency indices $(n,n)$ are given by
\[ \dom{T_{\alpha}} := \dom{T_{\text{min}}} + \bigvee_{i} \left ( - \phi_{i}^{+} + \sum_{j=1}^{n} \alpha_{i,j}\phi_{j}^{-}  \right),\]
where $\alpha \in \scr{U}_n$ is called the finite-rank perturbation parameter and $T_{\text{min}}$ is the minimal operator associated with $T$.
\end{defn}
Here, it is important to note that $\alpha$ is an operator mapping the negative defect space into the positive defect space. For fixed choices of bases for these two spaces, $\alpha$ has a matrix representation. Note that we do not symbolically distinguish between the operator and its matrix representation because for a given application the choice of bases will be made clear.

\subsection{Self-Adjoint Extensions of $C$-Symmetric Operators}
We begin by defining a space of $n \times n$ matrix-valued functions whose elements will be used to determine coefficient functions for quasi-differential equations. For an interval $J:=(a,b)$ with $- \infty < a < b \leq \infty$ and $n \in \N$, let
\[\begin{aligned}
Z_{n} (J) = &\bigg \{ Q= \left ( q_{r,s} \right )_{r,s=1}^{n}: Q \in  \left (L_{\text{loc}} (J) \right )^{n \times n}; \\
&\quad q_{r,r+1} \neq 0 \text{ a.e. on } J, q_{r,r+1}^{-1} \in L_{\text{loc}} (J), 1 \leq r \leq n-1;\\
&\quad q_{r,s} = 0 \text{ a.e. on } J, 2 \leq r+1 < s \leq n;\\
&\quad q_{r,s} \in L_{\text{loc}} (J), s \neq r+1,1 \leq r \leq n-1  \bigg \}.
\end{aligned}\]
Now, for $Q \in Z_{n} (J)$, quasi-derivatives are constructed in the following manner. First, let \[V_0 := \{f: J \rightarrow \C : f \text{ measurable} \}.\] 
We define the zeroth-order quasi derivative $f^{[0]}:=f$ for $f \in V_{0}$. Then, for $r=1, \dots, n$, we recursively define the spaces of quasi-differentiable functions  
\[V_r:= \left \{f \in V_{r-1} : f^{[r-1]} \in AC_{\text{loc}} (J) \right \}\]
and the higher order quasi-derivatives on those spaces
\[f^{[r]}:= q_{r,r+1}^{-1} \left ( \left ( f^{[r-1]}\right)^{\prime} - \sum_{s=1}^{r} q_{r,s}f^{[s-1]} \right ),\]
using the convention that $q_{n,n+1}=1$. With these quasi-derivatives, we now define the differential operator of interest  \[M_{Q} f := i^{n} f^{[n]}\]
for $f \in V_r$. We subscript $M$ by $Q$ to denote the fact that $M$ is constructed using the quasi-derivatives induced by a specific choice of $Q \in Z_{n} (J).$ 
\begin{remark}
\label{r-SLCYmmetric}
For $n=2$ and 
\[Q=\begin{pmatrix}
s & p \\
-q & -s 
\end{pmatrix},\] 
$M_{Q}=r\tau$ from Equation~\eqref{sturmliouville}. In this case, since there is a clean formula for the inverse of a $2 \times 2$ matrix and a wealth of theory developed for Sturm--Liouville operators, there exist alternate techniques for proving analogues to some of the following results (see \cite{EGNT2013} for details).
\end{remark}
\begin{remark}
\label{r-LnKnInTermsOfMQ}
Note that for $2 \leq n \in \N$ and the matrix $Q^0$ given by 
\[\begin{aligned}
    Q_{r,r+1}^{0}=1,& &r=1,2, \dots n-1,\\
    Q_{r,s}^{0}=0,& &1 \leq r \leq n-1, s \neq r+1.
\end{aligned}\]
all quasi-derivatives reduce to ordinary derivatives and \[M_{Q^{0}} f=i^n f^{(n)}\] for all $n$. Thus, for $0 < a \in \R$ and $J=(-a,a)$, $M_{Q^{0}}$ will correspond to $L_n$ as given in Definition~\ref{d-Ln}. Similarly, for $J=(0,\infty)$, $M_{Q^{0}}$ will correspond to $K_{n/2}$ as in Definition~\ref{d-Kn}. Furthermore, this $Q^0$ is $C$-symmetric for $C$ given by $C_{k,\ell}=(-1)^{\ell+1}\delta_{k,n+1-\ell}$, where $\delta$ is the Kronecker delta; that is, \[Q^0=-C^{-1} \left (Q^{0} \right )^{*}C.\]
\end{remark}

To begin our discussion of self-adjoint extensions of $M_Q$, we first define the maximal and minimal operators associated with $M_Q$ as in Definition~2 of \cite{WSZ2009}. Let $w \in L_{\text{loc}} (J)$ be positive a.e. on $J$. Then the maximal operator $S_{\text{max}}$ associated with $M_Q$ is given by $S_{\text{max}}f=M_Q f$ for $f$ in the set \[\dom{S_{\text{max}}} = \left \{ f \in L^2 (J,w) : f \in V_n, w^{-1} M_Q f \in L^2 (J,w) \right \}.\]
The minimal operator $S_{\text{min}}$ associated with $M_Q$ is the adjoint of the maximal operator $S_{\text{max}}$. 

An object closely tied to expressing the domains of self-adjoint extensions of $M_Q$ is the {\bf{Lagrange bracket}} $[\cdot,\cdot]$ defined for two functions $f,g \in V_n$ as \[[f,g]:= (-1)^{n/2} \sum_{r=0}^{n-1} (-1)^{n+1-r} \ov{g}^{[n-r-1]} f^{[r]}.\]
We will also denote the difference of the Lagrange bracket of two functions $f,g \in V_{r}$ evaluated at the endpoints of $J$ as \[ [f,g]_a^b := [f,g] (b) - [f,g] (a), \] where, as is true with classical derivatives, evaluation of functions at an endpoint of the interval $J$ is well-defined and can be determined by taking limits when the operator is regular.

The form of the self-adjoint boundary conditions will depend on whether the operator is regular, limit-point, or limit-circle at the endpoints of the interval $J$. In this context, the latter two terms have the same definition as in the classical setting (see Definition \ref{d-LcLpRankTwo}). Regularity, however, must be generalized:
\begin{defn}
\label{d-regularFiniteRank}
A $C$-Symmetric differential operator $M_Q$ is {\bf{regular}} near the endpoint $a$ if there exists $c \in (a,b)$ such that the following functions are in $L^1 (a,c)$, 
\begin{align*}
    &q_{r,r+1}^{-1}, \quad r=1,2, \dots, n-1\\
    &q_{r,s}, \quad 1 \leq r,s \leq n, s \neq r+1.
\end{align*}
Similarly, $M_Q$ is said to be regular near $b$ if there exists $d \in (a,b)$ such that the above functions are in $L^1 (d,b)$.
\end{defn}
\begin{remark}
As mentioned in Remark~\ref{r-LnKnInTermsOfMQ}, the $Q$ such that $M_Q$ corresponds to $L_n$ and $K_n$ consists entirely of entries that are constant functions. Thus, we have that $L_n$ is regular at both endpoints of its interval ($\pm a$ for some $0 < a \in \R$) and that $K_n$ is only regular at one endpoint of its interval. Specifically, $K_n$ is regular at 0 and is not regular, but is limit-circle, at $\infty$. We will proceed by first dealing with the case when an operator is regular at one endpoint and limit circle (but not necessarily regular) at the other to investigate boundary conditions for $K_n$. Then, we will restrict down to the case where both endpoints are regular to examine the boundary conditions for $L_n$. 
\end{remark}

Before we proceed to construct the domains of self-adjoint extensions of $M_Q$, we introduce the notion of patching, which will aid in future computation. 
\begin{lemma}[Extension of Naimark's Patching Lemma] 
\label{L-patchingNaimark}
Let $c,d \in \R$ such that $a \leq c < d \leq d$ and $\xi_0, \xi_1, \dots, \xi_{n-1}, \varphi_0, \varphi_1, \dots, \varphi_{n-1} \in \C$. Suppose $M_Q$ is regular on $(c,d)$. Then, there exists $f \in \dom{S_{\text{max}}}$ such that \[f^{[r]} (c)= \xi_r, \quad  f^{[r]} (d)= \varphi_r, \quad \text{for } r=1,2, \dots, n-1.\]
\end{lemma}

For the original statement of Naimark's Patching Lemma, see \cite[Lemma~4]{WSZ2009}. Using this patching result, we have an explicit expression for the minimal domain associated with $M_Q$.

\begin{lemma}
\label{L-EGNT3.6general}
The minimal domain associated with $M_Q$ is given by
\small 
\[\dom{S_{\text{min}}}:= \left \{ f \in  \dom{S_{\text{max}}}: [f,g](b)= [f,g](a)=0, \forall g \in \dom{S_{\text{max}}} \right \}.\]
\normalsize
\end{lemma}
\begin{proof}
Let $f \in \dom{S_{\text{min}}}$. For all $g \in \dom{S_{\text{max}}}$, integrating by parts in $L^2 (J,w)$ yields \[\ip{M_{Q}f}{g}_{L^2 (J,w)}-\ip{f}{M_{Q}g}_{L^2 (J,w)}= [f,g]_a^b .\] By the patching Lemma, for all $g \in \dom{S_{\text{max}}}$, there exists $g_a \in \dom{S_{\text{max}}}$ such that $g_a = g$ in a neighborhood of $a$ and  $g_a= 0$ in a neighborhood of $b$. Thus, we have \[ [f,g](a)=[f,g_a](a)=[f,g_a](b)=0. \] The same can be done for $g$ in a neighborhood of $b$. 

To show the backward direction, let $f \in \dom{S_{\text{max}}}$ such that \[[f,g](b)= [f,g](a)=0\] for all $g \in \dom{S_{\text{max}}}$. Then, \[\ip{M_{Q}f}{g}_{L^2 (J,w)}-\ip{f}{M_{Q}g}_{L^2 (J,w)}= [f,g]_a^b=0.\] Hence, $f \in \dom{S_{\text{max}}}^{*}=\dom{S_{\text{min}}}$ as desired.
\end{proof}

This patching lemma of Naimark is also integral to the proof of the following well-known characterization of self-adjoint extensions.

\begin{thm}[GKN Theorem]
\label{T-GKN}
Let $d$ be the deficiency index of $S_{\text{min}}$. An operator $S$ is a self-adjoint extension of $S_{\text{min}}$ if and only if there exists functions $v_k \in \dom{S_{\text{max}}}$ for $k=1,2, \dots, d$ such that,
\begin{equation}
    \label{e-CSymmetricSADomain2}
    \begin{aligned}
    (i)&\quad v_1, \dots, v_d \text{ are linearly independent modulo } \dom{S_{\text{min}}},\\
    (ii)&\quad [v_k  , v_\ell]_{a}^{b}=0 \text{ for all } k, \ell \in \{ 1, 2, \dots , d\}\\
    (iii)&\quad \dom{S}:= \{ f \in \dom{S_{\text{max}}} : [f,v_{k}]_a^b=0 \text{ for all } k  =1,2, \dots, d \}.
    \end{aligned}
\end{equation}
\end{thm}

\subsection{Self-Adjoint Domains in Terms of Boundary Conditions (Singular Endpoint)}
\label{ss-SABCHalfLine}
The first set of results of this subsection (Lemma~\ref{L-KnDecomp}, Theorem~\ref{T-3wsz2009}, and Theorem~\ref{T-5WSZ2009}) follows the derivation of Theorem 5 of \cite{WSZ2009} in order to arrive at the existence of a pair of matrices $A,B$ that parameterize the boundary conditions for any self-adjoint extension of $M_Q$ for the case when $M_Q$ is regular at $a$ and limit-circle at $b$. This is the case that includes our operator $K_n$. Then, we relate those matrices to the finite-rank perturbation parameter.  

The next two results establish the existence of functions that will be used to describe the self-adjoint extensions of $M_Q$ in the case of one limit-circle endpoint.

\begin{lemma}
\label{L-KnDecomp} 
Let $a_0 \in J$. Then, there exists functions $z_{k} \in \dom{S_{\text{max}}}$ for $k=1, \dots, n$ such that $z_k (x) =0$ for all $k=1, \dots, n$, $x \geq a_0$ and \[[z_k , z_ {\ell}] (a) = (-1)^{\ell+1}\delta_{k,n+1-\ell}.\]
\end{lemma}
\begin{proof}
This follows directly from applying Lemma~\ref{L-patchingNaimark} for $c=a$, $d=a_0$, 
\end{proof}
\begin{thm}
[{see, e.g.~\cite[Theorem 3]{WSZ2009}}] \label{T-3wsz2009}
Suppose that $M_Q$ is regular at the left endpoint $a$ and has deficiency index $d$. Let $m=2d-n$. Further suppose that there exists $\lambda_0 \in \R$ such that the eigenvalue equation \[ M_Q f = \lambda_0 w f \] has $d$ linearly independent solutions in $L^2 (J,w)$. Then, there exist functions $u_j$ for $j=1,2, \dots, m$ that solve $M_Q f = \lambda_0 w f$ for all $j$ and have the property that the matrix $U$ given by \[U_{j,k}:= [u_j, u_k] (a), \quad 1 \leq j,k \leq m\]
is non-singular and can be used to decompose the maximal domain in the following way
\[\dom {S_{\text{max}}} = \dom {S_{\text{min}}} \dot{+}  \spn{z_1, z_2, \dots, z_n} \dot{+} \spn{u_1, \dots, u_m},\]
where $z_k$ for $k=1,2, \dots, n$ are from Lemma~\ref{L-KnDecomp}.
\end{thm}

In order to elegantly express the boundary conditions that we will be using, we define three transformations that map $f \in V_n$ into vector-valued functions: \begin{equation}
\label{e-BCvectors}
    \hat{f} := \begin{pmatrix}[1.1]
f^{[0]}\\
f^{[1]}\\
\vdots \\
f^{[n-2]}\\
f^{[n-1]}
\end{pmatrix}, \quad \Check{f}:=\begin{pmatrix}[1.1]
-\ov{f}^{[n-1]}\\
\ov{f}^{[n-2]}\\
\vdots \\
-\ov{f}^{[1]}\\
\ov{f}^{[0]}
\end{pmatrix}, \quad \grave{f} := \begin{pmatrix}[1.1]
[f,u_1] \\
[f,u_2] \\
\vdots \\
[f,u_m]
\end{pmatrix}.
\end{equation}
In addition, we note that Theorem~\ref{T-3wsz2009} implies that the GKN functions have decomposition
\begin{equation}
\label{e-GKNfuncdecomp}
    v_{i}=\tilde{v}_{i}+\sum_{k=1}^{n} \zeta_{i,k} z_k + \sum_{j=1}^{n} e_{i,j} u_j,
\end{equation}
where $\tilde{v}_{i} \in \dom{S_{\text{min}}}$ and $\zeta_{i,k},  e_{i,j} \in \C$ are constants. Now, we are ready to discuss the boundary conditions:

\begin{thm}
[{see, e.g.~\cite[Theorem 5]{WSZ2009}}] \label{T-5WSZ2009}
Let $Q \in Z_{n} (J,\C)$ for $n$ even be $C$-symmetric for $C_{k,\ell}=(-1)^{\ell+1}\delta_{k,n+1-\ell}$. Suppose $M_{Q}$ is regular at $a$ and limit circle at $b$ (i.e. deficiency index $d=n$ and thus $m=n$). Then, an operator $S$ is a self-adjoint extension of $S_{\text{min}}$ if and only if there exists $ \beta_a , \beta_b \in \C^{n \times n}$ with the properties 
\begin{equation}
\label{e-CSymmetricSAConditionsLC}\text{rank} (\beta_a \vert \beta_b )=n \text{ and } \beta_a C \beta_a^{*} = \beta_b C \beta_b^{*}
\end{equation}  such that $S=S_{\text{min}}$ on the linear manifold 
\begin{equation}
\label{e-LCSADomain}
    \dom{S}:= \left \{ f \in \dom{S_{\text{max}}} : \beta_a \hat{f}(a) + \beta_b \grave{f} (b)=0 \right \},
\end{equation}
where $(\beta_a \vert \beta_b)$ is an augmented $n \times 2n$ matrix. For a fixed self-adjoint extension $S$ of $S_{\text{min}}$, one such pair of matrices is given by  \begin{equation}
    \label{E-ConstructAB2}
    \beta_a = \left ( \begin{array}{c}
         \Check{v}_{1}^{T} (a)\\
         \Check{v}_{2}^{T} (a)\\
         \vdots\\
         \Check{v}_{n}^{T} (a)
    \end{array} \right ), \quad \beta_b=\left ( \ov{e}_{i,j} \right )_{i,j=1}^{n},
\end{equation} where the vectors $\Check{v}_{k}^{T}$ are constructed from the GKN functions using the check transformation defined in Equation~\eqref{e-BCvectors} and $e_{i,j}$ are the same as in Equation~\eqref{e-GKNfuncdecomp}.
\end{thm}
Moving forward, we will label a self-adjoint extension of $S_{\text{min}}$ with the matrices $\beta_a,\beta_b$ that determine it, denoting it as $S (\beta_a ,\beta_b)$.

Next, we will connect the self-adjoint boundary conditions for $K_{n, \text{min}}$ to the finite-rank perturbation formulation for self-adjoint extensions of $K_{n, \text{min}}$. Recall, from Definition~\ref{d-genalpha}, that the self-adjoint extensions of $K_{n, \text{min}}$ are given by \[\dom{K_{n, \alpha}} := \dom{K_{n, \text{min}}} + \bigvee_{i} \left ( - \phi_{i}^{+} + \sum_{j=1}^{n} \alpha_{i,j}\phi_{j}^{-}  \right), \]
where $\phi_{k}^{\pm}$ for $k = 1, \dots, n$ are basis functions for the defect spaces $\mathcal{D}_{\pm}$ of $K_n$ and $\alpha \in \scr{U}_n$ is the finite-rank perturbation parameter.

\begin{thm}
\label{T_FRAlphaBCCOnnection2}
 $\dom{K_{n/2,\alpha}}=\dom{K_{n/2} (\beta_a,\beta_b )}$ for $\alpha$ that satisfies \[\begin{aligned}
     0 &= - \left (\begin{array}{c}
     \Check{v}_{1} \cdot \hat{\phi}_{i}^{+} (a)\\
     \vdots\\
     \Check{v}_{n} \cdot \hat{\phi}_{i}^{+} (a)
\end{array} \right ) + \sum_{j=1}^{n} \alpha_{i,j}  \left ( \begin{array}{c}
     \Check{v}_{1} \cdot \hat{\phi}_{j}^{-} (a)\\
     \vdots\\
     \Check{v}_{n} \cdot \hat{\phi}_{j}^{-} (a)
\end{array} \right ) \\
    &\quad - \left [ -\sum_{k=1}^{n}  \left (\begin{array}{c}
     \ov{e}_{1,k} [\phi_{i}^{+},u_{1}] (b)\\
     \vdots\\
     \ov{e}_{n,k} [\phi_{i}^{+},u_{n}] (b)
\end{array} \right ) + \sum_{j=1}^{n} \alpha_{i,j}  \sum_{k=1}^{n}  \left (\begin{array}{c}
     \ov{e}_{1,k} [\phi_{j}^{-},u_{1}] (b)\\
     \vdots\\
     \ov{e}_{n,k} [\phi_{j}^{-},u_{n}] (b)
\end{array} \right ) \right ]
     \end{aligned}\]
for all $i=1, \dots, n$.
\end{thm}

Theorem~\ref{T_FRAlphaBCCOnnection2} is a generalization of Proposition~\ref{p-PertBCR1HalfLine} from Section~\ref{s-prelim} to the finite-rank case.

\begin{proof}
Recall that $K_{n/2}=M_{Q_{0}}$, where $Q^0$ is the matrix given by \[ \begin{aligned}
    Q_{r,r+1}^{0}=1,& &r=1,2, \dots n-1,\\
    Q_{r,s}^{0}=0,& &1 \leq r \leq n-1, s \neq r+1.
\end{aligned} \] Further note that these functions are integrable near 0 and not integrable near $\infty$, thus $K_{n/2}$ is regular at $a=0$ and not regular at $b=\infty$. By Theorem~\ref{T-5WSZ2009}, since $K_{n/2,\alpha}$ is a self-adjoint extension of the minimal operator associated with $K_{n/2}$, denoted $K_{n/2, \text{min}}$, there exists matrices $\beta_a ,\beta_b \in \C^{n \times n}$ that meet the conditions~\eqref{e-CSymmetricSAConditionsLC} \[\text{rank} (\beta_a \vert \beta_b )=n \text{ and } \beta_a C \beta_a^{*} = \beta_b C \beta_b^{*}\] and define the domain of $K_{n/2,\alpha}$ as in Equation~\eqref{e-LCSADomain}, \[
\dom{K_{n/2,\alpha}} = \{ f \in \dom{K_{n/2, \text{max}}} : \beta_a \hat{f}(a) +\beta_b \grave{f} (b)=0\}.\] Note that if $\beta_a \hat{f}(a) +\beta_b \grave{f} (b)=0$, then the following matrix expression holds,
\[ (\beta_a \vert \beta_b ) \left ( \begin{array}{c}
     \hat{f}(a)\\
     \grave{f}(b)
\end{array} \right ) =0. \] 
Now let $f \in \dom{K_{n/2,\alpha}}$. Then, $f$ has the form
\[f= \tilde{f} + \sum_{i=1}^{n} d_i \left ( - \left (\phi_{i}^{+} \right )(a) + \sum_{j=1}^{n} \alpha_{i,j}\left (\phi_{j}^{-}\right )(a) \right),\]
where $\tilde{f} \in K_{n/2, \text{min}}$ and $d_i \in \C$ are constants. Using the linearity of quasi-differentiation and the Lagrange bracket, along with the fact that $\tilde{f} (a)=\tilde{f}(b)=0$, we have 
\[ \left ( \begin{array}{c}
     \hat{f}(a)\\
     \grave{f}(b)
\end{array} \right ) = \sum_{i=1}^{n} d_i \left ( - \left (\begin{array}{c}
     \hat{\phi}_{i}^{+}(a)\\
     \grave{\phi}_{i}^{+}(b)
\end{array} \right ) + \sum \alpha_{i,j}  \left ( \begin{array}{c}
     \hat{\phi}_{j}^{-}(a)\\
     \grave{\phi}_{j}^{-}(b)
\end{array} \right )\right ).\]
Next, we subsitute the definition of $\beta_a,\beta_b$ in terms of the GKN functions from Equation~\eqref{E-ConstructAB2}, 
\small
\begin{align*}
    0 &= \left ( \begin{array}{c|c}
          \Check{v}_{1}^{T} (a) & \ov{e}_{1,1} \hdots \ov{e}_{1,n} \\
         \Check{v}_{2}^{T} (a) & \ov{e}_{2,1} \hdots \ov{e}_{2,n}\\
         \vdots & \vdots\\
         \Check{v}_{n}^{T} (a) & \ov{e}_{n,1} \hdots \ov{e}_{n,n}
    \end{array} \right ) \left (\sum_{i=1}^{n} d_i \left [ - \left (\begin{array}{c}
     \hat{\phi}_{i}^{+}(a)\\
     \grave{\phi}_{i}^{+}(b)
\end{array} \right ) + \sum_{j=1}^{n} \alpha_{i,j}  \left ( \begin{array}{c}
     \hat{\phi}_{j}^{-}(a)\\
     \grave{\phi}_{j}^{-}(b)
\end{array} \right )\right ] \right )\\
&=\sum_{i=1}^{n} d_i \left [ - \left (\begin{array}{c}
     \Check{v}_{1} \cdot \hat{\phi}_{i}^{+} (a)\\
     \vdots\\
     \Check{v}_{n} \cdot \hat{\phi}_{i}^{+} (a)
\end{array} \right ) + \sum_{j=1}^{n} \alpha_{i,j}  \left ( \begin{array}{c}
     \Check{v}_{1} \cdot \hat{\phi}_{j}^{-} (a)\\
     \vdots\\
     \Check{v}_{n} \cdot \hat{\phi}_{j}^{-} (a)
\end{array} \right )\right ]\\
&\quad - \sum_{i=1}^{n} d_i \left [ -\sum_{k=1}^{n}  \left (\begin{array}{c}
     \ov{e}_{1,k} [\phi_{i}^{+},u_{1}] (b)\\
     \vdots\\
     \ov{e}_{n,k} [\phi_{i}^{+},u_{n}] (b)
\end{array} \right ) + \sum_{j=1}^{n} \alpha_{i,j}  \sum_{k=1}^{n}  \left (\begin{array}{c}
     \ov{e}_{1,k} [\phi_{j}^{-},u_{1}] (b)\\
     \vdots\\
     \ov{e}_{n,k} [\phi_{j}^{-},u_{n}] (b)
\end{array} \right ) \right ]
\end{align*}
\normalsize
where the first vertical bar in the line above denotes matrix augmentation. Note that here the choice of $f$ was arbitrary. Therefore, the only way that the above expression is true for all $f$ is if it holds true independent of the values of the constants $d_1, \dots, d_n$ (which depend on $f$). Namely, if
\begin{align*}
    0 &= - \left (\begin{array}{c}
     \Check{v}_{1} \cdot \hat{\phi}_{i}^{+} (a)\\
     \vdots\\
     \Check{v}_{n} \cdot \hat{\phi}_{i}^{+} (a)
\end{array} \right ) + \sum_{j=1}^{n} \alpha_{i,j}  \left ( \begin{array}{c}
     \Check{v}_{1} \cdot \hat{\phi}_{j}^{-} (a)\\
     \vdots\\
     \Check{v}_{n} \cdot \hat{\phi}_{j}^{-} (a)
\end{array} \right ) \\
    &\quad - \left [ -\sum_{k=1}^{n}  \left (\begin{array}{c}
     \ov{e}_{1,k} [\phi_{i}^{+},u_{1}] (b)\\
     \vdots\\
     \ov{e}_{n,k} [\phi_{i}^{+},u_{n}] (b)
\end{array} \right ) + \sum_{j=1}^{n} \alpha_{i,j}  \sum_{k=1}^{n}  \left (\begin{array}{c}
     \ov{e}_{1,k} [\phi_{j}^{-},u_{1}] (b)\\
     \vdots\\
     \ov{e}_{n,k} [\phi_{j}^{-},u_{n}] (b)
\end{array} \right ) \right ]
\end{align*}
for all $i=1, \dots, n$.
\end{proof}

\subsection{Self-Adjoint Domains in Terms of Boundary Conditions (Two Regular Endpoints)}
\label{ss-SABCInterval}
Similarly to the last subsection, we will first describe the boundary conditions of a class of operators, this time $M_Q$ such that $M_Q$ is regular at both endpoints of $J$. Then, we will connect these boundary conditions to the finite-rank perturbation parameter for our operator of interest $L_n$.

\begin{thm}
[{see, e.g.~\cite[Theorem 6]{WSZ2009}}] \label{T-1.1BSHZ2019}
Let $Q \in Z_{n} (J,\C)$ for $n$ even be $C$-symmetric for $C_{k,\ell}=(-1)^{\ell+1}\delta_{k,n+1-\ell}$. Suppose $M_{Q}$ is regular at both endpoints of $J$. Then, an operator $S$ is a self-adjoint extension of $S_{\text{min}}$ if and only if there exists $ \beta_a, \beta_b \in \C^{n \times n}$ with the properties
 \[\text{rank} (\beta_a \vert \beta_b )=n \text{ and } \beta_a C \beta_a^{*} = \beta_b C \beta_b^{*}\] such that $S=S_{\text{min}}$ on the linear manifold \begin{equation}
\label{e-CSymmetricSADomain}
\dom{S}:= \{ f \in \dom{S_{\text{max}}} : \beta_a \hat{f}(a) +\beta_b \hat{f} (b)=0\},\end{equation} where $(\beta_a \vert \beta_b)$ is an augmented $n \times 2n$ matrix. For a fixed self-adjoint extension $S$ of $S_{\text{min}}$, one such pair of matrices is given by \begin{equation}
    \label{E-ConstructAB}
    \beta_a= \left ( \begin{array}{c}
         \Check{v}_{1}^{T} (a)\\
         \Check{v}_{2}^{T} (a)\\
         \vdots\\
         \Check{v}_{n}^{T} (a)
    \end{array} \right ), \quad \beta_b=-\left ( \begin{array}{c}
         \Check{v}_{1}^{T} (b)\\
         \Check{v}_{2}^{T} (b)\\
         \vdots\\
         \Check{v}_{n}^{T} (b)
    \end{array} \right ).
\end{equation} where the vectors $\Check{v}_{k}^{T}$ are constructed from the GKN functions using the check transformation defined in Equation~\eqref{e-BCvectors}.
\end{thm}

As we did in the previous subsection, we will connect the boundary condition and finite-rank perturbation formulations for self-adjoint extensions of $L_{n, \text{min}}$. Using Definition~\ref{d-genalpha}, the self-adjoint extensions of $L_{n, \text{min}}$ described via finite-rank perturbations are given by \[\dom{L_{n, \alpha}} := \dom{L_{n, \text{min}}} + \bigvee_{i} \left ( - \phi_{i}^{+} + \sum_{j=1}^{n} \alpha_{i,j}\phi_{j}^{-}  \right), \]
where $\phi_{k}^{\pm}$ for $k = 1, \dots, n$ are basis functions for the defect spaces $\mathcal{D}_{\pm}$ of $L_n$ and $\alpha \in \scr{U}_n$ is the finite-rank perturbation parameter.

\begin{thm}
\label{T_FRAlphaBCCOnnection}
 $\dom{L_{n,\alpha}}=\dom{L_{n} (\beta_a,\beta_b )}$ for $\alpha$ that satisfies
\[0=\left. \left [ - \left (\begin{array}{c}
     \Check{v}_{1} \cdot \hat{\phi}_{i}^{+} (x)\\
     \vdots\\
     \Check{v}_{n} \cdot \hat{\phi}_{i}^{+} (x)
\end{array} \right ) + \sum_{j=1}^{n} \alpha_{i,j}  \left ( \begin{array}{c}
     \Check{v}_{1} \cdot \hat{\phi}_{j}^{-} (x)\\
     \vdots\\
     \Check{v}_{n} \cdot \hat{\phi}_{j}^{-} (x)
\end{array} \right )\right ] \right \vert_{x=b}^{a}\]
for all $i=1, \dots, n$.
\end{thm}

In a similar fashion to the main result of the previous subsection, Theorem~\ref{T_FRAlphaBCCOnnection} generalizes Proposition~\ref{p-PertBCR1Interval} from Section~\ref{s-prelim} to the finite-rank case.

\begin{proof}
The proof is this result is similar to the proof of Theorem~\ref{T_FRAlphaBCCOnnection2}, starting with invoking Theorem~\ref{T-1.1BSHZ2019} rather than Theorem~\ref{T-5WSZ2009} since $L_n$ is regular. The vector of boundary values is now
\[\left ( \begin{array}{c}
     \hat{f}(a)\\
     \hat{f}(b)
\end{array} \right ) = \sum_{i=1}^{n} d_i \left [ - \left (\begin{array}{c}
     \hat{\phi}_{i}^{+}(a)\\
     \hat{\phi}_{i}^{+}(b)
\end{array} \right ) + \sum \alpha_{i,j}  \left ( \begin{array}{c}
     \hat{\phi}_{j}^{-}(a)\\
     \hat{\phi}_{j}^{-}(b)
\end{array} \right )\right ].\]
Again, we substitute the definition of $\beta_a,\beta_b$ in terms of the GKN functions, this time for when the operator is regular at both endpoints of $J$ Equation~\eqref{E-ConstructAB}, 
\small
\begin{align*}
    0 &= \left ( \begin{array}{c|c}
          \Check{v}_{1}^{T} (a) & - \Check{v}_{1}^{T} (b)\\
         \Check{v}_{2}^{T} (a) & - \Check{v}_{2}^{T} (b)\\
         \vdots & \vdots\\
         \Check{v}_{n}^{T} (a) & -\Check{v}_{n}^{T} (b)
    \end{array} \right ) \left (\sum_{i=1}^{n} d_i \left [ - \left (\begin{array}{c}
     \hat{\phi}_{i}^{+}(a)\\
     \hat{\phi}_{i}^{+}(b)
\end{array} \right ) + \sum_{j=1}^{n} \alpha_{i,j}  \left ( \begin{array}{c}
     \hat{\phi}_{j}^{-}(a)\\
     \hat{\phi}_{j}^{-}(b)
\end{array} \right )\right ] \right )\\
&=\sum_{i=1}^{n} d_i \left. \left [ - \left (\begin{array}{c}
     \Check{v}_{1} \cdot \hat{\phi}_{i}^{+} (x)\\
     \vdots\\
     \Check{v}_{n} \cdot \hat{\phi}_{i}^{+} (x)
\end{array} \right ) + \sum_{j=1}^{n} \alpha_{i,j}  \left ( \begin{array}{c}
     \Check{v}_{1} \cdot \hat{\phi}_{j}^{-} (x)\\
     \vdots\\
     \Check{v}_{n} \cdot \hat{\phi}_{j}^{-} (x)
\end{array} \right )\right ] \right \vert_{x=b}^{a},
\end{align*}
\normalsize
where the first vertical bar in the line above denotes matrix augmentation. Since $f$ is again arbitrary, the above expression is true for all $f$ and is thus true independent of the values of the constants $d_1, \dots, d_n$ (which depend on $f$). This implies the conclusion of the theorem.
\end{proof}

\subsection{Connection to Characteristic Function Equivalence}
\label{ss-FRCharacteristic}
Recall from Section~\ref{s-matrixclark}, that two $B_{1}, B_{2}\in \scr{S}_n (\C_{+} )$ are said to be equivalent if there exist constant matrices  $R, Q \in \scr{U}_n$ such that \[B_1 (w) = R B_2 (w) Q \] for all $w \in \C_{+}$.  Theorem~\ref{t-LivsicEquivAC} says that $B_1$ and $B_2$ have Clark measures whose absolutely continuous parts are unitarily equivalent with a transformation of the perturbation parameter, \[ \frac{d \bmu^{B_1 \alpha^{*} }}{dm} (s)=R \frac{d \bmu^{B_2 (Q\alpha^{*}R)}}{dm} (s) R^{*},\]
for all $s \in \R$. Similarly, from Theorem~\ref{t-LivsicEquivPP}, we have that $B_1$ and $B_2$ have Clark measures whose pure-point parts are unitarily equivalent with that same transformation of the perturbation parameter, \[ \bmu^{B_1 \alpha^{*} }(\{s\})=R  \bmu^{B_2 (Q\alpha^{*}R)} (\{s\}) R^{*},
\] for all $s \in \R$. Liv{\v s}ic's theorem (Theorem~\ref{t-Livsic}) gives us that choosing a different bases for the defect spaces leads to equivalent Liv{\v s}ic characteristic functions. Thus, choosing different bases for the defect spaces leads to a unitarily equivalent spectral measure with a perturbation parameter conjugated by those matrices of equivalence. In particular, the spectral type of the corresponding scalar measures $\tr\mu$ is preserved, and Theorems~\ref{t-Livsic} and~\ref{t-LivsicEquivPP} show that a similar property holds for the matrix-valued measures as well.

Via our setup for finite rank perturbations, Definition~\ref{d-genalpha}, we have an explicit relationship for how changing the bases of the defect spaces affects the parameterization of the perturbation problem. Now suppose $B_1$ and $B_2$ are equivalent Liv{\v s}ic characteristic functions for the expression $T$ given by choosing bases $\left \{ \psi_{k}^{1} \right \}_{k=1}^{n}$ and $\left \{ \psi_{k}^{2} \right \}_{k=1}^{n}$, respectively, for the positive defect space of $T$. Then, the self-adjoint extension of $T$ corresponding to a fixed $\alpha_0 \in \scr{U}_n$ for the choice of basis $\left \{ \psi_{k}^{1} \right \}_{k=1}^{n}$ is unitarily equivalent to the self-adjoint extension of $T$ corresponding to $R^{*} \alpha_0 Q^{*}$ for the choice of basis $\left \{ \psi_{k}^{2} \right \}_{k=1}^{n}$. This change can be propagated to parameterizing self-adjoint extensions via boundary conditions to a differential expression $T$ through Theorems~\ref{T_FRAlphaBCCOnnection2} and~\ref{T_FRAlphaBCCOnnection}.


\section{Spectrum of Derivative Even Powers on the Half-Line}
\label{s-HalfLineSpectrum}
In this section, we will analyze the spectrum of powers of the derivative on the half-line. To this end, first recall the family of operators $K_n$ acting on $L^2 (0,\infty)$ from Definition~\ref{d-Kn}, \[K_n =  ( -1 )^{n} \frac{d^{2n}}{dx^{2n}}.\] Here, we present the absolutely continuous parts of the spectral measures of the self-adjoint extensions of two operators, a rank-one example ($K_1$) and a rank-two example ($K_2$). The singular parts of the decomposition of the spectrum of these operators are discussed by other authors. For example, see~\cite{Glazman} for information about the pure-point part (a negative eigenvalue) of the spectral measures of the self-adjoint extensions of $K_1$. 

\subsection{Second Derivative on the Half-line} \label{ss-HalfLineRankOne}
In order to obtain information about the spectral measure for $K_1$, we will use the results of Section~\ref{s-matrixclark}. We begin by determining the defect indices for the operator $K_1$. To this end, consider the eigenvalue equation \[K_1 f:= - f^{\prime \prime}  =wf\] on $L^2(0,\infty)$. This equation is solved by $f(w;x)=e^{i \sqrt{w} x}$ if $\Im (\sqrt{w}) > 0$ and $f(w;x)=e^{-i \sqrt{w} x}$ if $\Im (\sqrt{w}) < 0$. Since only $e^{i \sqrt{w} x}$ lands in $L^2(0,\infty)$ for the principal branch of the square root, for the bases of the defect spaces we have 
\[ \phi_{1,0}^{+} (w;x)=\phi_{1,0}^{-} (w;x)= e^{i \sqrt{w} x}.\] Thus, the defect indices of $K_1$ are (1,1). As was pointed out in Subsection~\ref{ss-Rob}, for this operator, the perturbation parameter $\alpha \in \scr{U}_1$ (i.e. a complex number on the unit circle). Recall that this parameter $\alpha$ determines the set of all self-adjoint extensions of $K_1$ as given by Definition~\ref{d-genalpha} and that this parameterization is equivalent to a choice of boundary conditions as shown in Proposition~\ref{p-PertBCR1HalfLine}.

Before we proceed with computing the analytic Gram matrix and the spectral measures for the self-adjoint extensions of $K_1$, we normalize the basis function for $\pm i$
\[ \tilde{\phi}_{1,0} (\pm i;x) = \sqrt[4]{2}e^{i \sqrt{\pm i} x}.\] 
Next, we compute the analytic Gram matrix (which is just a scalar in this case) using Equation~\eqref{EQdefA},
\begin{align*}
    A(w,\pm i) &=  \langle e^{i\sqrt{w} x}, \sqrt[4]{2}e^{i\sqrt{ \pm i} x}\rangle_{(0,\infty)} \\
    &= \sqrt[4]{2} \int_{0}^{\infty} e^{i\sqrt{w} x} \overline{e^{i\sqrt{ \pm i} x}} dx\\
    &= \left. \sqrt[4]{2} \frac{e^{\left (i\sqrt{w}+ \overline{ i\sqrt{ \pm i}}\right )x}}{i\sqrt{w}+ \overline{i\sqrt{ \pm i}}} \right \vert_{x=0}^{\infty}\\
    &=- \frac{\sqrt[4]{2}}{i\sqrt{w}+ \overline{ i\sqrt{ \pm i}}}\\
    &=\frac{\sqrt[4]{2}i}{\sqrt{w}- \overline{ \sqrt{ \pm i}}},
\end{align*}
where $\ip{\cdot}{\cdot}_{(0,\infty )}$ denotes the inner product on $L^2 (0,\infty)$. Then, utilizing Equation~\eqref{EQdefB}, the Liv{\v s}ic characteristic function is given by
\begin{align*}
    B(w) &= \frac{w-i}{w+i}\frac{\sqrt{w}- \overline{ \sqrt{i}}}{\sqrt{w}- \overline{ \sqrt{-i}}} \\
    &=\frac{(w-i)(\sqrt{w}- (1-i)/\sqrt{2})}{(w+i)(\sqrt{w}+ (1+i)/\sqrt{2})}.
\end{align*}
Now, we can compute the absolutely continuous part of the spectral measure.
\begin{thm}
\label{t-HalfLineRankOneSpectrum}
The Lebesgue weight of the absolutely continuous part of the Clark measure of $K_{1,\alpha}$ is given by
\[ \frac{d \mu^{B \alpha^{*}}}{dm} (s)= \frac{2\Re (\sqrt{2s} )}{\pi (\vert s \vert + 1 + 2 \Re (\sqrt{s} e^{-i \pi / 4}))\scr{D}(s)},\] where
\begin{equation}
\label{e-denomdef}
    \begin{aligned}
    \scr{D}(s) &:=2 (1-\Re (\alpha) ) s^2 + 2 \vert s \vert + \vert i \alpha -1 \vert^2 -2^{3/2} (\Re (\alpha ) -1) s^{3/2}\\
    &\quad +2 (1 - \Re (\alpha ) + \Im (\alpha) )s -2^{3/2} (\Im (\alpha) -1) s^{1/2}.
    \end{aligned}
\end{equation}
\end{thm}
\begin{proof}{}
Now, we can compute the Lebesgue weight of the spectral measure using Theorem~\ref{mtxAC},
\[ \frac{d \bmu^{B \alpha^{*} }}{dm} (s)= \frac{1}{\pi(1+s^2)} \lim _{w \downarrow s}  (\alpha^{*}-B^{*}(w))^{-1} (I-B^{*}(w)B(w)) (\alpha-B(w))^{-1}.\]
First, we examine $I-B^{*}B,$
\begin{align*}
    I-B^{*}(w)B(w) &=1-\frac{(\overline{w}+i)(\sqrt{\overline{w}}+e^{-i 3 \pi/4})}{(\overline{w}-i)(\sqrt{\overline{w}}+e^{-i \pi/4})} \frac{(w-i)(\sqrt{w}+e^{i 3 \pi/4})}{(w+i)(\sqrt{w}+e^{i \pi/4})} \\
    &=1-\frac{\vert w-i \vert^2}{\vert w+i \vert^2} \frac{\vert w \vert + 1 + 2 \Re (\sqrt{w} e^{-i 3\pi / 4})}{\vert w \vert + 1 + 2 \Re (\sqrt{w} e^{-i \pi / 4})}
\end{align*}
Letting $w$ tend to $s \in \R$,
\begin{equation}
    \label{e-K1proof1}
\begin{aligned}
    &\lim _{w \downarrow s} (I-B^{*}(w)B(w))\\
    &= 1-\frac{s^2+1}{s^2+1} \frac{\vert s \vert + 1 + 2 \Re (\sqrt{s} e^{-i 3\pi / 4})}{\vert s \vert + 1 + 2 \Re (\sqrt{s} e^{-i \pi / 4})}\\
    &=\frac{\left ( \vert s \vert + 1 + 2 \Re (\sqrt{s} e^{-i \pi / 4}) \right ) - \left ( \vert s \vert + 1 + 2 \Re (\sqrt{s} e^{-i 3\pi / 4})\right)}{\vert s \vert + 1 + 2 \Re (\sqrt{s} e^{-i \pi / 4})}\\
    &=\frac{2\Re (\sqrt{2s} )}{\vert s \vert + 1 + 2 \Re (\sqrt{s} e^{-i \pi / 4})}.
\end{aligned}
\end{equation}
Next, we determine $(\alpha^{*}-B^{*}(w))^{-1}$ and $(\alpha-B(w))^{-1}$. Taking advantage of the fact that we have commutation in the rank-one case, we can compute $\left \vert \alpha-B(w)  \right \vert^{-2}$. Note that since $\lim _{w \downarrow s} (I-B^{*}(w)B(w))$ is zero for $s \leq 0$ and $\left \vert \alpha-B(w)  \right \vert^{-2}$ is a power of a modulus, it follows that the Lebesgue weight itself is zero for $s \leq 0$. As such, we will assume that $s > 0$ for the rest of this proof.
We obtain
 \begin{align*}
     \left \vert \alpha-B(w)  \right \vert^{-2} &= \left \vert \alpha- \frac{w+\sqrt{2w}+1}{w+i}  \right \vert^{-2}\\
     &=\frac{\left \vert w+i \right \vert^2}{\left \vert (\alpha - 1 ) w - \sqrt{2w} +i \alpha -1 \right \vert^2}.
\end{align*}
The denominator of the above expands to
\begin{align*}
    &2 (1-\Re (\alpha) ) \vert w \vert^2 + 2 \vert w \vert + \vert i \alpha -1 \vert^2 -2^{3/2} \Re ( (\ov{\alpha}-1) \ov{w} \sqrt{w} ) \\
    &+2 \Re ((-i \ov{\alpha}-1)(\alpha-1) w ) -2^{3/2} \Re ( (-i \alpha -1) \sqrt{w} ).
\end{align*}
Substituting this in, then letting $w$ tend to $s \in \R$ yields
\begin{equation}
    \label{e-K1proof2}
\lim _{w \downarrow s} \left \vert \alpha-B(w)  \right \vert^{-2} = \frac{s^2 + 1}{\scr{D}(s)},
\end{equation}
where $\scr{D}(s)$ is as defined in Equation~\eqref{e-denomdef}. Multiplying Equations~\eqref{e-K1proof1} and \eqref{e-K1proof2} gives the desired result.
\end{proof}

\subsection{Fourth Derivative on the Half-Line} \label{ss-HalfLineRankTwo}
As with the previous section, we first examine the defect indices of our operator of interest, this time $K_2$. The eigenvalue equation 
\[K_2 f:= f^{(4)}  =wf\]
has solutions
\[f_{\ell}=\expp{\left (e^{i 2 \pi \ell} w \right ) ^{1/4} x}\]
for $\ell=1,2,3,4$. As is true in the previous section, only some of the eigenfunctions are in $L^2(0, \infty)$ and thus only some of them are in the maximal domain. Specifically, only $f_{1}, f_{2} \in L^2(0, \infty)$. Thus, $K_2$ has defect indices $(2,2)$ and $\alpha \in \scr{U}_2$. Again note that $\alpha$ specifies the self-adjoint extensions of $K_2$ via Definition~\ref{d-genalpha}. Now this parameterization has a correspondence with boundary conditions given by Theorem~\ref{T_FRAlphaBCCOnnection2}.

Our defect spaces have orthonormalized basis functions
\[ \tilde{\phi}_{2,k}^{\pm}(w;x)= \mathcal{C}_{k,1} \expp{iw^{1/4} x} + \mathcal{C}_{k,2} \expp{-w^{1/4} x}, \] where 
\begin{align*}
    \mathcal{C}_{1,1}&=\left ( 2 \Im (w^{1/4} ) \right )^{-1/2},\\
    \mathcal{C}_{1,2}&=0,\\
    \mathcal{C}_{2,1}&=\frac{{C}_{2,2}}{2 \Im (w^{1/4} ) (iw^{1/4}-\ov{w}^{1/4} )},\\
    \mathcal{C}_{2,2}&=\left ( \frac{1}{2 \Re (w^{1/4})}+\frac{1/8 - \Im (w^{1/4})^2}{\left ( \Re (w^{1/4} )+\Im (w^{1/4}) \right)^{2} \Im (w^{1/4})^3}\right )^{-1/2}.
\end{align*}
Employing Equation~\eqref{EQdefA} yields the matrix elements of the analytic Gram matrix,
\small
\begin{align*}
    &A (w , \pm i )_{j,k}\\
    &= \ip{\expp{\left (e^{i 2 \pi j} w \right ) ^{1/4} x} }{\mathcal{C}_{k,1} \expp{i(\pm i)^{1/4} x} + \mathcal{C}_{k,2} \expp{-(\pm i)^{1/4} x}}_{(0,\infty)}\\
    &= \ov{\mathcal{C}}_{k,1} \ip{\expp{\left (e^{i 2 \pi j} w \right ) ^{1/4} x}}{\expp{i(\pm i)^{1/4} x}}_{(0,\infty)}\\
    &\quad + \ov{\mathcal{C}}_{k,2} \ip{\expp{\left (e^{i 2 \pi j} w \right ) ^{1/4} x}}{\expp{-(\pm i)^{1/4} x}}_{(0,\infty)}\\
    &=\ov{\mathcal{C}}_{k,1} \int_{0}^{\infty} \expp{\left (e^{i 2 \pi j} w \right ) ^{1/4} x} \ov{\expp{i(\pm i)^{1/4} x}}dx \\
    &\quad + \ov{\mathcal{C}}_{k,2} \int_{0}^{\infty} \expp{\left (e^{i 2 \pi j} w \right ) ^{1/4} x} \ov{\expp{-(\pm i)^{1/4} x}}dx \\
    &=\ov{\mathcal{C}}_{k,1} \int_{0}^{\infty} \expp{\left (\left (e^{i 2 \pi j} w \right ) ^{1/4} -i (\mp i)^{1/4} \right ) x}dx \\
    &\quad + \ov{\mathcal{C}}_{k,2} \int_{0}^{\infty} \expp{\left (\left (e^{i 2 \pi j} w \right ) ^{1/4} - (\mp i)^{1/4} \right ) x}dx \\
    &= \frac{\ov{\mathcal{C}}_{k,1}}{i (\mp i)^{1/4}- (e^{i 2 \pi j} w) ^{1/4}}+\frac{\ov{\mathcal{C}}_{k,2}}{(\mp i)^{1/4}- (e^{i 2 \pi j} w) ^{1/4}}.
\end{align*}
\normalsize
Again, we compute the Liv{\v s}ic characteristic function with Equation~\eqref{EQdefB}, this time utilizing the formula for the inverse of a $2 \times 2$ matrix
\begin{equation} \label{B2matrix} B(w) = \frac{w-i}{g(w)} \begin{pmatrix}
    A_{2,2}^{+}A_{1,1}^{-}- A_{1,2}^{+}A_{2,1}^{-} & A_{2,2}^{+}A_{1,2}^{-}- A_{1,2}^{+}A_{2,2}^{-} \\
    -A_{2,1}^{+}A_{1,1}^{-}+ A_{1,1}^{+}A_{2,1}^{-} & -A_{2,1}^{+}A_{1,2}^{-}+ A_{1,1}^{+}A_{2,2}^{-}
\end{pmatrix},\end{equation} where we have abbreviated $A(w,\pm i)_{j,\ell}$ as $A_{j,\ell}^{\pm}$ and \begin{equation}  \label{B2denominator} g(w)=(w+i)(A_{1,1}^{+}A_{2,2}^{+}- A_{1,2}^{+}A_{2,1}^{+}).\end{equation}
From Theorem~\ref{mtxAC}, we have that the Lebesgue weight of our Clark measure has the form 
\[ \frac{d \bmu^{B \alpha^{*} }}{dm} (s)= \frac{1}{\pi(1+s^2)} \lim _{w \downarrow s}  (\alpha^{*}-B^{*}(w))^{-1} (I-B^{*}(w)B(w)) (\alpha-B(w))^{-1}.\]
Throughout the rest of this argument, we suppress the dependence of functions on independent variable $s$. To further simplify our notation, let
\[ M:= \alpha- B , \quad N:=I -B^{*}B, \quad \kappa :=\frac{1}{\pi(1+s^2)}. \]
Then, the Lebesgue weight becomes
\[ \frac{d \bmu^{B \alpha^{*} }}{dm} = \kappa \left (M^{*} \right )^{-1}NM^{-1}.\]
Writing the Lebesgue weight in terms of matrix elements yields
\begin{align*}
    \frac{d \bmu^{B \alpha^{*} }}{dm}&=\kappa \left ( \begin{pmatrix} 
           M_{1,1} & M_{1,2}\\
           M_{2,1} & M_{2,2}\\
        \end{pmatrix}^{-1} \right )^{*} \begin{pmatrix} 
           N_{1,1} & N_{1,2}\\
           N_{2,1} & N_{2,2}\\
        \end{pmatrix} \begin{pmatrix} 
           M_{1,1} & M_{1,2}\\
           M_{2,1} & M_{2,2}\\
        \end{pmatrix}^{-1}\\
    &=\frac{\kappa}{\vert \det (M) \vert^2}  \begin{pmatrix} 
           \ov{M}_{2,2} & -\ov{M}_{2,1}\\
           -\ov{M}_{1,2} & \ov{M}_{1,1}\\
        \end{pmatrix}  \begin{pmatrix} 
           N_{1,1} & N_{1,2}\\
           N_{2,1} & N_{2,2}\\
        \end{pmatrix} \begin{pmatrix} 
           M_{2,2} & -M_{1,2}\\
           -M_{2,1} & M_{1,1}\\
        \end{pmatrix}.
\end{align*}
Thus, the matrix elements of the Lebesgue weight are,
\small
\begin{align}
    \label{e-Rank2HalfLineProof01} \left [ \frac{d \bmu^{B \alpha^{*} }}{dm} \right]_{1,1}&=\frac{\vert M_{2,2}\vert^2 N_{1,1}-M_{2,1}\ov{M}_{2,2}N_{1,2}-\ov{M}_{1,2}M_{2,2}N_{2,1}+N_{2,2}\vert M_{2,1} \vert^{2}}{\vert \det (M) \vert^2/\kappa},\\
    \label{e-Rank2HalfLineProof02} \left [ \frac{d \bmu^{B \alpha^{*} }}{dm} \right]_{1,2}&=\frac{\ov{M}_{2,2}(M_{1,1}N_{1,2}-M_{1,2}N_{1,1})+\ov{M}_{2,1}(M_{1,2}N_{2,1}-M_{1,1}N_{2,2})}{\vert \det (M) \vert^2/\kappa},\\
    \label{e-Rank2HalfLineProof03} \left [ \frac{d \bmu^{B \alpha^{*} }}{dm} \right]_{2,1}&=\frac{\ov{M}_{1,2}(M_{2,1}N_{1,2}-M_{2,2}N_{1,1})+\ov{M}_{1,1}(M_{2,2}N_{2,1}-M_{2,1}N_{2,2})}{\vert \det (M) \vert^2/\kappa},\\
    \label{e-Rank2HalfLineProof04} \left [ \frac{d \bmu^{B \alpha^{*} }}{dm} \right]_{2,2}&=\frac{\vert M_{1,2}\vert^2 N_{1,1}-M_{1,1}\ov{M}_{1,2}N_{1,2}-\ov{M}_{1,1}M_{1,2}N_{2,1}+N_{2,2}\vert M_{1,1} \vert^{2}}{\vert \det (M) \vert^2/\kappa}.
\end{align}
\normalsize
Next, we compute the matrix elements for $N$, \begin{equation}
    \label{e-Rank2HalfLineProof05}
    N= \begin{pmatrix}
    1- \vert B_{1,1} \vert^2 - \vert B_{2,2} \vert^2& \ov{B}_{1,1} B_{1,2} + \ov{B}_{2,1} B_{2,2} \\
    B_{1,1} \ov{B}_{1,2} + B_{2,1} \ov{B}_{2,2} & 1- \vert B_{1,2} \vert^2 - \vert B_{2,1} \vert^2
    \end{pmatrix}.
\end{equation} 
Substituting Equations~\eqref{e-Rank2HalfLineProof05} and $M=\alpha-B$ into Equations~\eqref{e-Rank2HalfLineProof01}-\eqref{e-Rank2HalfLineProof04} yields the Lebesgue weight of the Clark measure.

\section{Spectrum of Derivative Powers on an Interval}
\label{s-IntervalSpectrum}
We will now examine the spectra of self-adjoint extensions of the following family of operators $L_n$ acting on $L^2(-a,a)$ as given in Definition~\ref{d-Ln}, 
\[L_n= i^n \frac{d^n}{dx^n}. \] As we did on the half-line, we will use the method detailed in \cite{AMR} to obtain an analytic Gram matrix, Liv{\v s}ic characteristic function, and ultimately spectra information about the self-adjoint extensions of these operators. 

\begin{remark}
Powers of the derivative on a finite interval will only have pure-point spectra. To see this, first note that one can use tools from classical ordinary differential equations theory to express the corresponding resolvent operators as integral operators with square-integrable kernels, thus making the resolvents Hilbert--Schmidt operators. Further computation will show that these resolvent operators have continuous inverses on the entire Hilbert space $L^2 (-a,a)$. Using the invertibility of the resolvent operators, one can then show that the original operators of interest (power of the derivative on a finite interval) have only pure-point spectrum. This computation is carried out explicitly for the rank-one case in~\cite{BirmanSolomyak1987}.

\end{remark}

\subsection{First Derivative on Interval} \label{ss-IntervalRankOne}
Analogously to what was done in Subsections~\ref{ss-HalfLineRankOne} for $K_1$ and~\ref{ss-HalfLineRankTwo} for $K_2$, the first step in determining spectral information about the self-adjoint extensions of $L_1$ is to determine its deficiency indices. The eigenvalue equation \[L_1 f:= i f^{\prime}  =wf\] is solved on $L^2(-a,a)$ by $f(w;x)=e^{-iwx}.$ Thus, we have
\[ \phi_{1,1} (w;x) = e^{-iwx}, \quad \tilde{\phi}_{1,1} (\pm i ;x) = \frac{e^{\pm x}}{\sqrt{2a}}.\] Therefore, $L_1$ has deficiency indices $(1,1)$ and $\alpha \in \scr{U}_1 = \mathbb{T}$. In other words, the self-adjoint extensions of $L_1$ are parameterized by a choice of a complex number on the unit circle via Definition~\ref{d-genalpha}. Furthermore, also recall that this parameterization is equivalent to boundary conditions as in Proposition~\ref{p-PertBCR1Interval}. Now, we compute the analytic Gram matrix $A$ as outlined before (using Equation~\eqref{EQdefA}), noting that the matrix is a scalar since we are in a rank-one setting,
\begin{align*}
    A(w,\pm i) &= \frac{1}{\sqrt{2a}} \langle e^{-iwx}, e^{\pm x} \rangle_{(-a,a)} \\
    &= \frac{1}{\sqrt{2a}} \int_{-a}^{a} e^{-iwx} e^{\pm x} dx\\
    &= \left. \frac{1}{\sqrt{2a}} \frac{e^{(\pm 1 -iw)x}}{\pm 1 -iw} \right \vert_{x=-a}^{a}\\ 
    &= \frac{1}{\sqrt{2a}} \frac{\sinh ((\pm 1 -iw)a)}{\pm 1 -iw} \\
    &= \frac{1}{\sqrt{2a}} \frac{\sin ((w \pm i)a)}{w \pm i},
\end{align*}
where $\ip{\cdot}{\cdot}_{(-a,a )}$ denotes the inner product on $L^2 (-a,a)$. Next, we compute the Liv{\v s}ic characteristic function using Equation~\eqref{EQdefB},
\begin{align*}
    B(w) &= \frac{w-i}{w+i} \frac{\sqrt{2a}(w+i)}{\sin ((w+i)a)} \frac{\sin ((w-i)a)}{\sqrt{2a}(w-i)}\\
    &=  \frac{\sin ((w-i)a)}{\sin ((w+i)a)}.
\end{align*}
Note that the cancellation of the Blaschke factor in $B$ is an artifact of this specific operator and does not happen in general. 

Now that we have an expression for the Liv{\v s}ic characteristic function, we can compute some spectral results.

\begin{thm}\label{thcarrierinterval1}
The location of the point masses of the Clark measure of $L_{1,\alpha}$ are the set $$P_{\alpha} = \left \{s=\frac{1}{a} \tan^{-1} \left ( \frac{\alpha^{*}+1}{\alpha^{*}-1} \tan (ia) \right) + \frac{n \pi}{a}  : n \in \mathbb{Z} \right \},$$
where we extend the definition of the inverse tangent (i.e. $\tan^{-1} (\pm \infty) \equiv \pm \pi / 2$). 
\end{thm} 
\begin{proof}
We begin by applying Nevanlinna's formula for $s \in \mathbb{R}$,  
$$\bmu^{B \alpha ^* } \left( \{ s \} \right) = \frac{2 i}{\pi (1+s^2 )^2} \lim _{w \downarrow s} (s-w ) \left (1 - \frac{\sin ((w-i)a)}{\sin ((w+i)a)} \alpha ^* \right ) ^{-1}.$$
From this, we have that there can only be point spectrum when
\begin{equation}
\nonumber
    \lim _{w \downarrow s} \frac{\sin ((w-i)a)}{\sin ((w+i)a)} = \frac{1}{\alpha ^*}.
\end{equation}
In the above limit, since the numerator and denominator cannot simultaneously be zero, we can evaluate the limit and simplify using trigonometric identities,
\begin{equation}
\label{ppIntervalCondition}
    (\alpha ^* + 1)\tan(ia)=(\alpha ^* - 1)\tan(sa).
\end{equation}
Thus, the set of point spectrum for the measure $\bmu^{B \alpha ^* }$ is given by,
$$P_{\alpha} = \left \{s= \frac{1}{a} \tan^{-1} \left ( \frac{\alpha^{*}+1}{\alpha^{*}-1} \tan (ia) \right) + \frac{n \pi}{a}  : n \in \mathbb{Z} \right \}.$$ Note that for $\alpha=1,$ the above definition is well defined by considering the extended real line and taking limits to define the inverse tangent of $\pm \infty $.
\end{proof}
Following this, we immediately verify an explicit intertwining property for the point spectra of Clark measures $\bmu^{B \alpha ^* }$ related to the Liv{\v s}ic characteristic function $B$. 
\begin{cor}
\label{IntervalPPMono}
The location of the point masses of the Clark measure of $L_{1,\alpha}$ are monotonically increasing with respect to the argument of $\a$. 
\end{cor}

\begin{proof}
Let $\a = e^{i \theta}.$ Then, we have $$s( \theta)= \frac{1}{a} \tan^{-1} \left ( \frac{e^{-i \theta}+1}{e^{-i \theta}-1} \tan (ia) \right) + \frac{n \pi}{a}.$$ Computing the derivative yields,
\begin{align*}
    \frac{ds}{d \theta} &= \frac{1}{a} \frac{1}{1+ \left( \frac{e^{-i \theta}+1}{e^{-i \theta}-1} \tan (ia) \right)^{2}}\left[2i \tan(ia) \frac{e^{-i \theta}}{(e^{-i \theta}-1)^{2}} \right]\\
    &= \frac{2}{a} \frac{1}{1+ \left( \frac{e^{-i \theta}+1}{e^{-i \theta}-1} \tan (ia) \right)^{2}}\left[ \frac{e^{-a}-e^{a}}{e^{-a}+e^{a}} \frac{1}{(e^{i \theta/2}-e^{-i \theta/2})^{2}} \right]\\
    &= -\frac{1}{2a} \frac{1}{1+ \left( \frac{e^{-i \theta}+1}{e^{-i \theta}-1} \tan (ia) \right)^{2}} \frac{e^{-a}-e^{a}}{e^{-a}+e^{a}} \csc^{2} \left (\frac{\theta}{2} \right).
\end{align*}
All that is left is to verify the sign of the derivative.
\begin{align*}
    \sgn \left(\frac{ds}{d \theta} \right ) &= \sgn \left( -\frac{1}{2a} \frac{1}{1+ \left( \frac{e^{-i \theta}+1}{e^{-i \theta}-1} \tan (ia) \right)^{2}} \frac{e^{-a}-e^{a}}{e^{-a}+e^{a}} \csc^{2} \left (\frac{\theta}{2} \right) \right)\\
    &=- \sgn \left ( \frac{1}{1+ \left( \frac{e^{-i \theta}+1}{e^{-i \theta}-1} \tan (ia) \right)^{2}} \right ) \sgn \left ( \frac{e^{-a}-e^{a}}{e^{-a}+e^{a}} \right )\\
    &= \sgn \left ( 1+ \left( \frac{e^{-i \theta}+1}{e^{-i \theta}-1} \tan (ia) \right)^{2}\right )
\end{align*}
yielding the result.
\end{proof}

Next, we will compute the masses of the atoms of the Clark measure.
\begin{thm}
\label{th1}
The weight of the Clark measure of $L_{1,\alpha}$ on the set $P_{\alpha}$ for $\alpha \neq \pm 1$ is given by,
$$\bmu^{B \alpha ^* } \left( \{ s \} \right) = 
-\frac{1}{2 \pi a \Im (\alpha)} \frac{\sin (2sa)}{(1+s^2)^2}.$$
For $\alpha=1$ and $\alpha=-1$, the point spectra have weights
$\bmu^{B} \left( \{ s \} \right) = \frac{2 \coth(a)}{ a \pi (1+s^2 )^2}$ and $\bmu^{-B} \left( \{ s \} \right) = \frac{2  \tanh (a)}{ a \pi (1+s^2 )^2},$ respectively.
\end{thm}

\begin{remark}
It is not difficult to see that the measures $\bmu^{B,\alpha}$ are Poisson summable. For example, $$\int\ci\R d\bmu^{-B}(s)= \frac{2  \tanh (a)}{ a \pi } \sum_{n\in\mathbb{Z}}\frac{a^4}{(a^2 + \pi^2 n^2)^2}<\infty.$$
\end{remark} 
\begin{proof}
For $s \in P_{\a}$, we apply L'Hospital's rule to compute the weight of the measure:
\begin{equation}
    \label{IntervalmeasurePP1}
    \begin{aligned}
        \bmu^{B \alpha ^* } \left( \{ s \} \right) &= \frac{2 i}{\pi (1+s^2 )^2} \lim _{w \downarrow s} (-1) \left (- a\frac{\sin(2ia)}{\sin^{2} ((w+i)a)} \alpha ^* \right ) ^{-1} \\
    &=\frac{2 i}{\alpha ^* a\pi (1+s^2 )^2} \frac{\sin^{2} (sa+ia)}{\sin(2ia)}  \\
    &= \frac{2 i}{\alpha ^* a \pi (1+s^2 )^2} \left [ \frac{\sin^{2}(sa) \cos (ia)}{\sin (ia)}+2\sin (sa) \cos (sa) \right. \\
    & \left. \quad \quad \quad \quad \quad \quad \quad \quad \quad+\frac{\cos^{2}(sa) \sin (ia)}{\cos (ia)}\right ].
    \end{aligned}
\end{equation}
From here, the computation varies dependent on $\a$.  First examine the case of $\alpha \neq \pm 1$, applying the condition in Equation~\eqref{ppIntervalCondition} to further simplify Equation~\eqref{IntervalmeasurePP1},
\small
\begin{align*}
    \bmu^{B \alpha ^* } \left( \{ s \} \right) &= \frac{2 i}{\alpha ^* a \pi (1+s^2 )^2} \left [ \sin^{2}(sa) \cot(sa) \frac{\alpha ^* + 1}{\alpha ^* - 1}+2\sin (sa) \cos (sa) \right. \\
    & \left. \quad \quad \quad \quad \quad \quad \quad \quad \quad+ \cos^{2}(sa) \tan (sa) \frac{\alpha ^* - 1}{\alpha ^* + 1}\right ]\\
    &=  \frac{i}{\alpha ^* a \pi } \frac{\sin (2sa)}{(1+s^2)^2} \left [ \frac{\alpha ^* + 1}{\alpha ^* - 1}+2+ \frac{\alpha ^* - 1}{\alpha ^* + 1}\right ]\\
    &=  \frac{i}{\alpha ^* a \pi } \frac{\sin (2sa)}{(1+s^2)^2} \frac{(\alpha ^*)^{2}}{(\alpha ^*)^{2} - 1}\\
    &= \frac{i}{ a \pi} \frac{\sin (2sa)}{(1+s^2)^2} \frac{1}{-2i \Im (\alpha)}\\
    &= -\frac{1}{2 \pi a \Im (\alpha)} \frac{\sin (2sa)}{(1+s^2)^2}.
\end{align*}
\normalsize
For $\alpha=1$, we have $s=(n+1/2) \pi /a$ for $n \in \mathbb{Z}$. Thus, we can simplify the trigonometric functions in Equation~\eqref{IntervalmeasurePP1}, $$\bmu^{B} \left( \{ s \} \right)= \frac{2 i \cot (ia)}{ a \pi (1+s^2 )^2} =\frac{2 \coth(a)}{ a \pi (1+s^2 )^2} .$$

If $\alpha=-1$, we analogously have, $$\bmu^{-B} \left( \{ s \} \right)=-\frac{2 i \tan (ia)}{ a \pi (1+s^2 )^2} = \frac{2  \tanh (a)}{ a \pi (1+s^2 )^2}.$$
\end{proof}

As a spectral measure, the Clark measure must be nonnegative. However, that is not immediately obvious from the formula in Theorem~\ref{th1}. We now verify this nonnegativity.

\begin{prop}
The Clark measure $\bmu^{B \alpha ^* }$ is nonnegative.
\end{prop}
\begin{proof}
Note that the nonnegativity of the measure for $\a = \pm 1$ comes immediately from Theorem~\ref{th1}.

Using the product definition of the sine function, \[ \sin (x) = x \prod_{k=1}^{\infty} \left [ 1 - \left( \frac{x}{k \pi} \right)^{2} \right], \] our measure becomes \[\bmu^{B \alpha ^* } \left( \{ s \} \right)=-\frac{s}{ \pi \Im (\alpha)(1+s^2)^2} \prod_{k=1}^{\infty} \left [ 1 - \left( \frac{2sa}{k \pi} \right)^{2} \right]. \]Observe that the nonnegativity of this function depends on the signs of $\Im (\a)$ and $s$ and the number of negative terms in the product. We will now show that the number of negative terms in this product is finite and depends on the signs of $\Im (\a)$ and $s$, leaving us with only four cases to consider rather than eight.

For the moment fix $\a \in \T$ with $\Im (\a) >0$ and further suppose that $s>0$. Since $s$ is monotone with respect to the argument of $\a$, there exists an odd number $N \in \N$ such that \[ N \frac{\pi}{2a} < s < (N+1) \frac{\pi}{2a}.\] Noting that the terms of the product are increasing in the index variable $k$ and that the $k$-th term of the product is negative when $| s | > k \pi / (2a)$, we have that an odd number (specifically the first N) of terms of the product are negative. Therefore, the product is negative, which implies that $\bmu^{B \alpha ^* }$ is nonnegative on the set $P_\alpha\cap [0,\infty)$ and therefore nonnegative on $[0,\infty)$ by Theorem~\ref{thcarrierinterval1}.

Analogous logic can be used to determine the number of negative terms in the product for the other three cases for the sign of signs of $\Im (\a)$ and $s$. Specifically, the product will have an even number of negative terms if one of, but not both, $\Im (\a)$ and $s$ are negative and the product will have an odd number of negative terms if both $\Im (\a)$ and $s$ are negative.
\end{proof}

\subsection{Second Derivative on Interval}
As was done in previous examples, we start the process of obtaining spectral information about the self-adjoint extensions of $L_2$ by determining its deficiency indices. To do so, we first look at the eigenvalue equation \[L_2 f:= - f^{\prime \prime}  =wf\] on $L^2(-a,a)$. This has two solutions, $f(w;x)=e^{i \sqrt{w} x}$ and $f(w;x)=e^{-i \sqrt{w} x}$. Therefore, we have basis functions for the defect spaces given by
\[ \phi_{2,k} (w;x) = e^{(-1)^{k} i \sqrt{w} x},\]  for $k=1,2$. Specifically, substituting $w \in \C_{+}$ into $\{ \phi_{2,k}\}_{k=1}^{2} $ will give a basis for $\mc{D} _+$ and doing the same with $w \in \C_{-}$ will give a basis for $\mc{D} _-$. Hence, $L_2$ has deficiency indices $(2,2)$, which further implies $\alpha \in \scr{U}_{2}$. Note that this parameterization is equivalent to applying boundary conditions via Theorem~\ref{T_FRAlphaBCCOnnection}.
Next, we generate orthonormalized bases for $w=\pm i$ yields
\[ \tilde{\phi}_{2,k} (\pm i;x) = \mathcal{C}_{k,1} e^{i \sqrt{\pm i}x}+\mathcal{C}_{k,2} e^{-i \sqrt{\pm i}x}, \]
where
\[\begin{array}{rclrcl}
    \mathcal{C}_{1,1}&=& \sqrt{\frac{\sqrt{2}}{\sinh \left( \sqrt{2}a \right)}},& \mathcal{C}_{1,2}&=&0,\\
    \mathcal{C}_{2,1}&=&\left( \frac{1}{\mathcal{C}_{1,1}^2} \left ( 1+ \frac{1}{\sinh (a)}\right) - 2 \sqrt{\sinh (a)}\right )^{-1/2},& \mathcal{C}_{2,2}&=&- \frac{\mathcal{C}_{2,1}}{\sqrt{\sinh (a)}}.
\end{array}\]
Again, we will use Equation~\eqref{EQdefA} to compute the elements of the analytic Gram matrix. Taking advantage of symmetries in the fact that all four matrix entries have a common form, we have
\begin{align*}
    A(w,\pm i)_{j,\ell} &= \langle e^{(-1)^{j} i \sqrt{w} x} , \mathcal{C}_{\ell,1} e^{i \sqrt{\pm i}x}+\mathcal{C}_{\ell,2} e^{-i \sqrt{\pm i}x} \rangle_{(-a,a)} \\
    &= \overline{\mathcal{C}}_{\ell,1}\langle e^{(-1)^{j} i \sqrt{w} x} ,  e^{i \sqrt{\pm i}x} \rangle_{(-a,a)}\\
    &\quad + \overline{\mathcal{C}}_{\ell,2}\langle e^{(-1)^{j} i \sqrt{w} x} ,  e^{-i \sqrt{\pm i}x} \rangle_{(-a,a)}\\
    &=\overline{\mathcal{C}}_{\ell,1} \int_{-a}^{a} e^{\left ( (-1)^{j} i \sqrt{w} + \overline{i \sqrt{\pm i}}\right)x} dx\\
    &\quad +\overline{\mathcal{C}}_{\ell,2} \int_{-a}^{a} e^{\left ( (-1)^{j} i \sqrt{w} + \overline{-i \sqrt{\pm i}}\right)x} dx\\
    &=\overline{\mathcal{C}}_{\ell,1} \frac{\sinh \left ( \left ( (-1)^{j} i \sqrt{w} + \overline{i \sqrt{\pm i}}\right)a \right )}{(-1)^{j} i \sqrt{w} + \overline{i \sqrt{\pm i}}}\\
    &\quad + \overline{\mathcal{C}}_{\ell,2} \frac{\sinh \left ( \left ( (-1)^{j} i \sqrt{w} + \overline{-i \sqrt{\pm i}}\right)a \right )}{(-1)^{j} i \sqrt{w} + \overline{-i \sqrt{\pm i}}}.
\end{align*}
Since this is a rank-two case, the Liv{\v s}ic characteristic function is given by Equation~\eqref{B2matrix}-\eqref{B2denominator}. That is, 
\[ B(w) = \frac{w-i}{g(w)} \begin{pmatrix}
    A_{2,2}^{+}A_{1,1}^{-}- A_{1,2}^{+}A_{2,1}^{-} & A_{2,2}^{+}A_{1,2}^{-}- A_{1,2}^{+}A_{2,2}^{-} \\
    -A_{2,1}^{+}A_{1,1}^{-}+ A_{1,1}^{+}A_{2,1}^{-} & -A_{2,1}^{+}A_{1,2}^{-}+ A_{1,1}^{+}A_{2,2}^{-}
\end{pmatrix},\]
where 
\[ g(w)=(w+i)(A_{1,1}^{+}A_{2,2}^{+}- A_{1,2}^{+}A_{2,1}^{+}),\] with $A(w,\pm i)_{j,\ell}$ abbreviated as $A_{j,\ell}^{\pm}$.
This leads us to the following spectral result.
\begin{thm}
\label{t-IntervalRankTwoSpectrum}
The spectral measure of the operator $L_{2,\alpha}$ has point spectra at $s \in \R$ if the matrix $I-B(w) \alpha^{*}$ has nontrivial kernel in the limit as $w \in \C_{+}$ approaches $s$ non-tangentially.
\end{thm}
\begin{proof}
This follows almost immediately from Equation~\eqref{uhpnev} provided that $B(w)$ is defined in the limit. $B$ is well-defined in this limit because the denominators of $A(w,\pm i)_{j,\ell}$ are never zero since $(-1)^{j+1} i \sqrt{s}$ is either purely imaginary (if $s>0$), purely real (if $s<0$), or zero (if s=0) and $\overline{\pm i \sqrt{\pm i}}$ always has both both real and imaginary part nonzero. 

Also notice that \eqref{B2matrix} implies that the case $\lim _{w \downarrow s}  g(w)=0$ will not produce any eigenvalues.
\end{proof}


\bibliography{ClarkTh}
\end{document}